\documentclass[12pt,oneside,reqno]{amsart}
\usepackage[colorlinks,linkcolor=blue,citecolor=red]{hyperref}
\usepackage{amssymb}
\usepackage{}
\usepackage{bbm}
\usepackage{cases}
\usepackage{amsmath}
\usepackage{graphicx}
\usepackage{mathrsfs}
\usepackage{stmaryrd}
\usepackage{amsfonts}
\usepackage{enumerate,amsmath,amssymb,amsthm}

\numberwithin{equation}{section}

\newcommand{\be}{\begin{eqnarray}}
\newcommand{\ee}{\end{eqnarray}}
\newcommand{\ce}{\begin{eqnarray*}}
\newcommand{\de}{\end{eqnarray*}}
\newtheorem{theorem}{Theorem}[section]
\newtheorem{lemma}[theorem]{Lemma}
\newtheorem{remark}[theorem]{Remark}
\newtheorem{definition}[theorem]{Definition}
\newtheorem{proposition}[theorem]{Proposition}
\newtheorem{Examples}[theorem]{Example}
\newtheorem{corollary}[theorem]{Corollary}
\newtheorem{assumption}{Assumption}

\def\bt{\begin{theorem}}
\def\et{\end{theorem}}
\def\bl{\begin{lemma}}
\def\el{\end{lemma}}
\def\br{\begin{remark}}
\def\er{\end{remark}}
\def\bx{\begin{Examples}}
\def\ex{\end{Examples}}
\def\bd{\begin{definition}}
\def\ed{\end{definition}}
\def\bp{\begin{proposition}}
\def\ep{\end{proposition}}
\def\bc{\begin{corollary}}
\def\ec{\end{corollary}}
\def\ba{\begin{assumption}}
\def\ea{\end{assumption}}
\def\bpf{\begin{proof}}
\def\epf{\end{proof}}

\def\cB{{\mathcal B}}
\def\cC{{\mathcal C}}

\def\cF{{\mathcal F}}

\def\cH{{\mathcal H}}

\def\cS{{\mathcal S}}
\def\cT{{\mathcal T}}

\def\mC{{\mathbb C}}
\def\mD{{\mathbb D}}
\def\mE{{\mathbb E}}

\def\mI{{\mathbb I}}

\def\mP{{\mathbb P}}

\def\mR{{\mathbb R}}
\def\mS{{\mathbb S}}

\def\R{\mathbb R}
\def\N{\mathbb N}

\def\bP{{\mathbf P}}

\def\bP{{\mathbf P}}

\def\bE{{\mathbf E}}
\def\1{{\mathbf{1}}}

\def\sA{{\mathscr A}}
\def\sB{{\mathscr B}}
\def\sC{{\mathscr C}}

\def\sF{{\mathscr F}}

\def\sK{{\mathscr K}}
\def\sL{{\mathscr L}}
\def\sM{{\mathscr M}}

\def\sP{{\mathscr P}}

\def\sS{{\mathscr S}}

\def\eps{\varepsilon}
\def\d{\text{\rm{d}}}

\def\e{\mathrm{e}}

\def\a{\alpha}
\def\om{\omega}
\def\Om{\Omega}

\def\p{\partial}

\def\si{\sigma}

\def\({{\Big(}}
\def\){{\Big)}}
\def\[{{\Big[}}
\def\]{{\Big]}}
\def\<{{\langle}}
\def\>{{\rangle}}

\def\bx{{\mathbf{x}}}

\def\sgn{\textit{\rm sgn}}

\def\dif{{\mathord{{\rm d}}}}

\def\={&\!\!=\!\!&}

\def\geq{\geqslant}
\def\leq{\leqslant}

\def\iint{\int_0^t\!\!\!\int}

\allowdisplaybreaks

\begin{document}
\title{Nonlocal elliptic equation in H\"older space and the martingale problem}
\author{Chengcheng Ling and Guohuan Zhao}

\address{Chengcheng Ling:
School of science, Beijing Jiaotong University, Beijing 100044, China\\
Faculty of Mathematics, Bielefeld University,
33615, Bielefeld, Germany\\
Email: cling@math.uni-bielefeld.de
 }

\address{Guohuan Zhao:
Faculty of Mathematics, Bielefeld University,
33615, Bielefeld, Germany\\
Email: zhaoguohuan@gmail.com
}

\thanks{
Research of C. Ling is  supported  by IRTG 2235 Bielefeld-Seoul ¡°Searching for the regular in the
irregular: Analysis of singular and random systems¡±.\\
\indent Research of G. Zhao is supported by the German Research Foundation (DFG) through the Collaborative Research Centre(CRC) 1283 ``Taming uncertainty and profiting from randomness and low regularity in analysis, stochastics and their applications".
}

\begin{abstract}
The well-posedness of  nonlocal elliptic equation with singular drift is investigated in Besov-H\"older spaces. As an application, we show the existence and uniqueness for corresponding martingale problem. Moreover, we prove that the one dimensional distribution of the martingale solution has a density in some Besov space.


\end{abstract}

\maketitle
\section{Introduction and Main Result}
We consider the following nonlocal elliptic equation in $\R^d$:
\be\label{Eq-PDE}
\lambda u-\sL^\a_\kappa u-b\cdot \nabla u=f.
\ee
Here $\a\in (0,2)$,  $b\in \sC^\beta$(Besov-H\"older space, see Definition \ref{Def-Besov} below) with $\beta\in \R$,  $\kappa$ is a nonnegative measurable function from $\R^d\times\R^d$ to $[0,\infty)$ and
$$
\sL_{\kappa}^\a f(x):=\int_{\R^d} \left(f(x+z)-f(x)-\nabla f(x)\cdot z^{(\a)}\right) \frac{\kappa(x, z)}{|z|^{d+\alpha}}\d z,
$$
where $z^{(\a)}:=z\1_{\{|z|<1\}}\1_{\a=1}+z\1_{\a\in(1,2)}$.

The first aim of our work is to establish a Schauder's type estimate for the solution to  \eqref{Eq-PDE} with irregular coefficients. There are many literatures studied this problem in different settings. When $\a\in (1,2)$, $b$ is a H\"older continuous function and $\sL_\kappa^\a$ is some $\a$-stable type  operator,  Priola in \cite{Pr1} and \cite{Pr2} studied the a priori estimate by using classic perturbation argument. Similarly, Athreya, Butkovsky and Mytnik in \cite{At-Bu-My} showed the global estimate for $\sL^\a_\kappa=\Delta^{\a/2}$ with $\a\in (1,2)$ and $b\in \sC^\beta$ with $\beta>\frac{1-\alpha}{2}$. Indeed, the analytic result in \cite{At-Bu-My} also holds for any non degenerate $\a$-stable operators. For $\a>1$, in \cite{Mi-Pr},  Mikulevicius and Pragarauskas also studied  the nonlocal Cauchy problem with first order term in H\"older space. And recently, in \cite{Do-Ji-Zh}, Dong, Jin and Zhang studied the Dini and Schauder estimate for nonlocal fully nonlinear equations. However, when $\a<1$, both \cite{Mi-Pr} and \cite{Do-Ji-Zh} must assume  $b\equiv0$.   To our best knowledge, when $\a\in (0,1)$, the interior estimate for the solution to \eqref{Eq-PDE} with non divergence free drift was first  obtained by Silvestre in \cite{Si2}. He used the extension method for $\sL^\a_{\kappa}=\Delta^{\a/2}$ when $\a\in (0,1)$ and $b\in \sC^\beta$ with $\beta>1-\a$ to reduce the nonlocal problem to the local case.   Recently, similar result was extended for stable-like operators in \cite{Zh-Zh3} by using Littewoord-Paley theory. Let us also mention that there are much more works for nonlocal equation without first order term, for instance  \cite{Ba-Ka}, \cite{Do-Ki} and the references therein.  

\medskip
In this work, we will show the global estimates  in more general setting.  Our assumption on $\kappa$ is:
\ba\label{Aspt1}
There are constants $r_0, \Lambda_1, \Lambda_2, \Lambda_3>0$, $\vartheta\in (0,1)$ such that
\begin{align}\label{AS1}
\int_{B_r} \kappa(x, z)\,\d z\geq \Lambda_1r^d, \quad x\in \R^d, r\in (0,r_0]
\tag{\bf H$_1$};
\end{align}
\begin{align}\label{AS2}
\kappa(x, z)\leq \Lambda_2 , \ x,z\in \R^d; \ \ \1_{\alpha=1}\int_{\{r<|z|<R\}} z\cdot\kappa(x,z)\dif z=0,\ 0<r<R<\infty
\tag{\bf H$_2$};
\end{align}
\begin{align}\label{AS3}
|\kappa(x,z)-\kappa(y,z)|\leq \Lambda_3 |x-y|^{\vartheta}, \quad  x,y\in\R^d, \vartheta\in (0,1)
\tag{\bf H$_3$}.
\end{align}
\ea

\iffalse
Although our the results seem quite expectable, one will see that the problems are not trivial, especially when the conditions on the kernel $\kappa$ are weak. For instance, the estimates for perturbation terms will be more delicate when $\a>1$ and $\kappa(x, z)$ is H\"older in $x$ and measurable in $z$(see Lemma \ref{Le-Lf-Var}).
\fi

\medskip

The following is our first main result:
\bt\label{Th-Holder-Var}
Suppose $\kappa(x,z)$ satisfies \eqref{AS1}-\eqref{AS3} and $\max\{{0, (1-\a)}\}<\vartheta<1$.
\begin{enumerate}
\item If $\a\in(0,1]$, $\beta\in (1-\a, \vartheta)$ and $b\in \sC^\beta$, then there are constants $\lambda_0, C>0$ such that for any $\lambda\geq \lambda_0$ and $f\in \sC^\beta$, equation \eqref{Eq-PDE} has a unique solution in $\sC^{\a+\beta}$ satisfying
\be\label{Eq-Holder-Var}
(\lambda-\lambda_0)\|u\|_{\sC^\beta}+\|u\|_{\sC^{\a+\beta}}\leq C \|f\|_{\sC^\beta},
\ee
where $\lambda_0, C$ only depend on $d, \a, \beta, \vartheta, r_0, \Lambda_1, \Lambda_2, \Lambda_3, \|b\|_{\sC^\beta}$.

\item If $\a\in(1,2)$, $\beta\in (-(\frac{\a-1}{2}\wedge \vartheta), \vartheta)$ and $b\in \sC^\beta$, then the above  conclusions also hold.
\end{enumerate}
\et
Notice that our condition \eqref{AS1} is much weaker than the usual lower bounded assumption $\kappa(x,z)\geq \lambda>0$ and also weaker than Assumption A(i) in \cite{Mi-Pr}. This makes our proof  A typical example is take
$$
\kappa(x,z)=\1_{V(x)}(z).
$$
Here $V(x)\in \R^d$ is a conical set of the form $V(x)=\{z\in \R^d: |\<z/|z|,\xi(x)\>|> \delta\}$ with measurable $\xi: \R^d \to \mS^{d-1}$, and $\delta>0$ is fixed.\\

Like in \cite{Zh-Zh3}, our approach of getting the Schauder type estimate is based on Littlewood-Paley theory. For the first case in Theorem \ref{Th-Holder-Var}, the key step is to  establish a frequency localized maximum inequality(see Lemma \eqref{Le-Max1} below). This kind of maximum principle  appeared in \cite{Wa-Zh} for $\kappa\equiv1$. We extend their result for any $\kappa(x,z)=\kappa(z)$ satisfying \eqref{AS4} below.   When $\a>1$ and $\beta\in (-(\frac{\a-1}{2}\wedge \vartheta), 0]$, the main problem is how to prove the boundedness of $\sL_\kappa^\a: \sC^{\a+\beta}\to \sC^{\beta}$, where the Bony's decomposition  plays a crucial rule in our proof.

\medskip

As one of the motivations of considering the regularity estimate for \eqref{Eq-PDE}, we want to investigate the well-posedness of the following SDE in $\R^d$:
\be\label{Eq-SDE}
X_t=X_0+\int_0^t \sigma(X_{s-}) \d Z_s+\int_0^t  b(X_s)\d s
\ee
in weak sense. Here $Z_t$ is an $\a$-stable process in $\R^d$, $\si$ is a $d\times d$-matrix-valued measurable function and $b$ is the drift, which might be very singular. Suppose $Z_t$ is rotational symmetric, $L_\sigma^\a+b\cdot\nabla $ is the generator of $X_t$, for any $\sigma$ satisfies \eqref{AS-sigma1} below, we have
\begin{align*}
L^\a_\sigma f(x)+b\cdot \nabla f(x)
=&\int_{\R^d} (f(x+\sigma(x)z)-f(x)-\nabla f(x)\cdot \sigma(x)z^{(\a)})\frac{\d z}{|z|^{d+\a}}+b\cdot\nabla f(x)\\
=&\int_{\R^d} (f(x+z)+f(x)-\nabla f(x)\cdot z^{(\a)}) \frac{\d z}{|\det \sigma(x)|\cdot|\sigma^{-1}(x)z|^{d+\a}}\\
&+b\cdot \nabla f(x) =\sL_{\kappa,b}^\a f(x),
\end{align*}
where
\be\label{Eq-Ka-Ro}
\kappa(x,z):=\frac{|z|^{d+\a}}{|\det \sigma(x)|\cdot|\sigma^{-1}(x)z|^{d+\a}}.
\ee
Since the  well-posedness of  the resolvent equations or backward Kolmogorov equations associated with $L_\sigma^\a+b\cdot \nabla $ are closely related to the  weak solutions(or martingale solutions) of \eqref{Eq-SDE}, our analytic result Theorem \ref{Th-Holder-Var} has direct applications to SDE driven by $\a$-stable process.

\medskip

On the other hand, pathwise uniqueness and strong existence for \eqref{Eq-SDE} with irregular coefficients have already been studied in a large number of literatures, see \cite{Ta-Ts-Wa} for one dimensional case and  \cite{Pr1}, \cite{Zh1}, \cite{Pr2}, \cite{Ch-So-Zh}, \cite{Ch-Zh-Zh}, etc for more general L\'evy noises in $\R^d$.  Roughly speaking, these works showed that the SDE \eqref{Eq-SDE} has a unique strong solution under the conditions that $\si$ is bounded, uniformly nondegenerate and Lipschitz, $Z_t$ is a non degenerated $\a$-stable process, $b\in \sC^\beta$ with $\beta>1-\frac{\a}{2}$. However, when we consider the existence and uniqueness  of weak solutions to \eqref{Eq-SDE} or the well-posedness of corresponding martingale problem, the regularity assumptions on the coefficients can be  released. In \cite{Zh-Zh1}, the authors considered \eqref{Eq-SDE} driven by Brownian motion, they showed that if $\sigma=\mI$, $b\in H^{-\frac{1}{2}}_p$ with $p>2d$ one can still give a natural meaning of $``\int_0^t b(X_s) \d s"$(see also \cite{Zh-Zh2}). The drift term may not be a process with finite variation any more but an additive functional of $X$ with zero energy. In \cite{At-Bu-My}, they considered the similar SDEs driven by one dimensional additive $\alpha$-stable noise with singular drifts in Besov-H\"older space. The above works are motivated by Bass and Chen's  early works \cite{Ba-Ch1}, \cite{Ba-Ch2}.

\medskip

In this paper, we will study the martingale problem associated with $\sL^\a_{\kappa, b}:=\sL_{\kappa}^\a+b\cdot\nabla$. When $\a\leq 1$, since we assume $b\in \sC^\beta$ with $\beta>0$, there is no issue about the definition of martingale or weak solution. However, when $\a>1$ and $b\in \sC^\beta$ with $\beta\leq 0$, like in \cite{Zh-Zh1}, \cite{At-Bu-My}, we need to give an appropriate  definition of solutions to \eqref{Eq-SDE}(see Definition \ref{Def-MP}). Combining Theorem \ref{Th-Holder-Var} and some standard techniques in probability theory, we can obtain the following result. We distribute the proof in Lemma \ref{Le-WP1} and Lemma \ref{Le-WP2}.
\bc\label{Cor-MP}
Suppose $\max\{{0, (1-\a)}\}<\vartheta<1$, $\kappa(x,z)$ satisfies \eqref{AS1}-\eqref{AS3}, and $b\in \sC^\beta$, where
$\beta\in (1-\a, \vartheta)$ if $\a\in(0,1]$ and $\beta\in (-(\frac{\a-1}{2}\wedge \vartheta), 0]$ if $\a\in(1,2)$. Then, for each $x\in \R^d$, there is a unique probability measure $\mP_x$ with starting point $x$ on the Skorokhod space $\mD$, which solves the martingale problem associated with $\sL^\a_{\kappa, b}$ and satisfies the Krylov's type estimate(see Definition \ref{Def-Kry}).
\ec

Our corollary above implies:
\bp
Suppose $Z_t$ is a rotational symmetric $\a$-stable process, $\sigma$ satisfies
\be\label{AS-sigma1}
\Lambda^{-1} |z|\leq |\sigma(x) z|\leq \Lambda |z|,\quad \Lambda>0,  z\in \R^d.
\ee
then
\begin{enumerate}[(i)]
\item If  $\a\in (0,1]$, $\beta\in (1-\a, 1)$. $\sigma, b\in \sC^\beta$, there is a unique weak solution to \eqref{Eq-SDE}.  %and $X_t$ has a density $p_t(x, \cdot)\in B^\gamma_{q,\infty}$, where $\gamma, q$ satisfy .
\item If $\a\in (1,2)$, $\beta\in (\frac{1-\a}{2}, 0]$, $\eps>0$. $\sigma\in \sC^{-\beta+\eps}$, $b\in \sC^{\beta}$,  there is a unique weak solution to \eqref{Eq-SDE}. %and $X_t$ has a density $p_t(x, \cdot)\in B^\gamma_{q,\infty}$, where $\gamma, q$ satisfy .
\end{enumerate}
\ep

\medskip

Another interesting problem we attempt to study in this paper is the regularity estimates for the one dimensional distribution of the solutions to martingale problem associated with $\sL^\a_{\kappa,b}$. Debussche and Fournier in \cite{De-Fo} proved that the law of the solution to \eqref{Eq-SDE} has a density in some Besov space, under some non-degeneracy condition on the driving L\'evy process and some H\"older-continuity assumptions on the coefficients. %However, they did not study the existence and uniqueness of the weak solution to \eqref{Eq-SDE}.
We following the thoughts in \cite{De-Fo}, but instead of using the crucial  Lemma 2.1 therein, we use the Littlewood-Paley description of Besov spaces to simplify the proof and get a bit more general result(see Lemma \ref{Le-Density-Y}).
%Notice that in \cite{De-Fo}, the authors did not study the existence and uniqueness of the weak solution to \eqref{Eq-SDE}, we think our contributions in this direction are: (i) we give sufficient conditions that ensure the well-posedness of \eqref{Eq-SDE};  (ii) we extend the main result in \cite{De-Fo} and simplify its proof(see Lemma \ref{Le-Density-R} for details).

\bt\label{Th-Density}
Under the same conditions in Corollary \ref{Cor-MP} for each $x\in \R^d$, suppose  $\mP_x$  is the unique solution in Corollary \ref{Cor-MP}. Then, for each $t>0$ the distribution of canonical process $\om_t$ under $\mP_x$ has a density in Besov space $B^\gamma_{q,\infty}$ with $\gamma$ and $q$ satisfying
 $$
0<\gamma < \a(\a +\beta -1), \quad 1\leq q<\frac{d}{d+\gamma-\a(\a+\beta-1)}
$$
  if $\alpha\in(0,1]$
  and
  $$
0<\gamma < (\a +\beta -1)\wedge \tfrac{\vartheta}{\a}, \quad 1\leq q<\frac{d}{d+\gamma-(\a+\beta-1)\wedge \frac{\vartheta}{\a}}
$$
  if $\alpha\in(1,2)$.
\et

Let $\sP(\R^d)$ be the collection of all probability measures on $\R^d$. Combining corollary \ref{Cor-MP} and Theorem \ref{Th-Density}, we obtain the following interesting corollary:

\bc\label{Cor-FPE}
Under the same conditions in Corollary \ref{Cor-MP} for any $x\in \R^d$, the following nonlocal Fokker-Planck equation:
\be\label{Eq-FPE}
\<\varrho_t, \phi\> = \phi(x)+\int_0^t \< \varrho_s, \sL_{\kappa, b}^\a \phi\> \dif s, \quad \forall \phi\in C_c^\infty
\ee
 has a {\em unique} solution $\{\varrho_t\}\subseteq \sP(\R^d)$. Moreover, for each $t>0$, $\varrho_t\in B^\gamma_{q,\infty}$ with $\gamma$ and $q$ satisfying
 $0<\gamma < \a(\a +\beta -1), \quad 1\leq q<\frac{d}{d+\gamma-\a(\a+\beta-1)}$ if $\alpha\in(0,1]$ and $0<\gamma < (\a +\beta -1)\wedge \tfrac{\vartheta}{\a}, \quad 1\leq q<\frac{d}{d+\gamma-(\a+\beta-1)\wedge \frac{\vartheta}{\a}}$ if $\alpha\in(1,2)$.
\ec

\br
The above result can also be seen as a probabilistic approach to the theory of regularity of solutions to non-local partial differential equations. We give a probabilistic proof for the well-posedness as well as regularity estimates for linear Fokker-Plank equation with singular coefficients and initial data.  %There are many work in this field in the PDE community. A typical result show that when the initial condition is bounded, then the viscosity solution to \eqref{Eq-FPE} is H\"older continuous. See Barles, Chasseigne, and Imbert \cite{Ba-Ch-Im} and the references therein.
\er

\medskip

This paper is organized as follows: In Section 2, we recall some basic knowledge from Littlewood-Paley theory for later use. We establish apriori estimates for \eqref{Eq-PDE} in H\"older-Besov spaces in Section 3. In Section 4, we  prove the well-posedness of martingale problem associated with $\sL_{\kappa, b}^\a$.  In section 5, we show the one dimensional distribution of the martingale solution has a density in some Besov space. 
%For the completeness of the paper we add the equivalence between martingale solution and weak solution in Appendix.

\section{Perliminary}
In this section, we recall some basic concepts and properties of Littlewood-Paley decomposition that will be used later.

Let $\sS(\R^d)$ be the Schwartz space of all rapidly decreasing functions, and $\sS'(\R^d)$ the dual space of $\sS(\R^d)$. Given $f\in\sS(\R^d)$,
let $\sF f=\hat f$  be the Fourier transform of $f$ defined by
$$
\hat f(\xi):=(2\pi)^{-d/2}\int_{\mR^d}\e^{-\mathrm{i}\xi\cdot x} f(x)\dif x.
$$

Let $\chi:\mR^d\to[0,1]$ be a smooth radial function with
$$
\chi(\xi)=1,\ |\xi|\leq 1,\ \chi(\xi)=0,\ |\xi|\geq 3/2.
$$
Define
$$
\cC:=B_{3/2}\backslash \overline{B_{1/2}}=\{x\in\R^d: 1/2< |x|<3/2\}; \quad \varphi(\xi):=\chi(\xi)-\chi(2\xi).
$$
It is easy to see that $\varphi\geq 0$ and supp $\varphi\subset \cC$ and
\begin{align}\label{EE1}
\chi(2\xi)+\sum_{j=0}^k\varphi(2^{-j}\xi)=\chi(2^{-k}\xi)\stackrel{k\to\infty}{\to} 1.
\end{align}
In particular, if $|j-j'|\geq 2$, then
$$
\mathrm{supp}\varphi(2^{-j}\cdot)\cap\mathrm{supp}\varphi(2^{-j'}\cdot)=\emptyset.
$$
In this paper we shall fix such $\chi$ and $\varphi$ and also introduce another nonnegative function $\tilde \varphi\in C_c^\infty(\R^d)$ supported on $B_2\backslash \overline{B_{1/4}}$ and $\tilde \varphi=1$ on $\cC$ for later use.

We introduce the definition of Besov space below.
\bd\label{Def-Besov}
The dyadic block operator $\Delta_j$ is defined by
$$
\Delta_j f:=
\left\{
\begin{array}{ll}
\sF^{-1}(\chi(2\cdot) \sF f), & j=-1, \\
\sF^{-1}(\varphi(2^{-j}\cdot) \sF f),& j\geq 0.
\end{array}
\right.
$$
For $s\in\mR$ and $p,q\in[1,\infty]$, the Besov space $B^s_{p,q}$ is defined as the set of all $f\in\sS'(\R^d)$ with
$$
\|f\|_{B^s_{p,q}}:=\left(\sum_{j\geq -1}2^{jsq}\|\Delta_j f\|_p^q\right)^{1/q}<\infty;
$$
If $p=q=\infty$, we denote $\sC^s:= B^s_{\infty, \infty}$.
\ed
Let
$$
h:=\sF^{-1} \varphi, \quad \tilde h:=\sF^{-1} \tilde\varphi,\quad  h_{-1}:=\sF^{-1} \chi(2\cdot);
$$
\begin{align*}
h_j:=\sF^{-1}\varphi(2^{-j}\cdot)=2^{jd}h(2^j\cdot),  \quad j\geq 0.
\end{align*}

By definition it is easy to see that
\begin{align}\label{Eq-Delta-f}
\Delta_j f(x)=(h_j*f)(x)=\int_{\mR^d}h_j(x-y)f(y)\dif y,\ \ j\geq -1.
\end{align}
\bd
The low-frequency cut-off operator $S_{j}$ is defined by
$$
{S_{j} f:=\sum_{j^{\prime} \leq j-1} \Delta_{j^{\prime}} f}.
$$
The  paraproduct of $f$ by $g$ is defined by
$$
T_{f} g:= \sum_{j\geq-1} S_{j-1} f \Delta_{j} g.
$$
The remainder of $f$ and $g$ is defined by
$$
R(f, g)=\sum_{|k-j| \leq 1} \Delta_{k} f \Delta_{j} g.
$$
\ed

The following two Lemmas can be found in \cite{Tr}.
\bl
If $s>0, s\notin \N$, then
$$
\sC^s:= B^s_{\infty, \infty}\asymp C^s,
$$
where $C^s$ is the usual H\"older space.
\el

\bl[Bernstein's inequalities]\label{Le13}
For any $1\leq p\leq q\leq \infty$ and $j\geq 0,$ we have
\begin{align}\label{Eq-BI1}
\|\nabla^k\Delta_j f\|_q\leq C_p 2^{(k+d(\frac{1}{p}-\frac{1}{q}))j}\|\Delta_jf\|_p,\  k=0,1,\cdots,
\end{align}
and
\begin{align}\label{EE99}
\|(-\Delta)^{s/2}\Delta_j f\|_q\leq C_p 2^{(s+d(\frac{1}{p}-\frac{1}{q}))j}\|\Delta_jf\|_p, \quad s\geq 0.
\end{align}
\el

\section{Schauder Estimate for \eqref{Eq-PDE}}
In this section, we establish the Schauder type estimate for \eqref{Eq-PDE} and its well-posedness in Besov-H\"older space.

\subsection{The case $\kappa(x,z)=\kappa(z)$} The following assumptions will be needed in this subsection.

\ba\label{Aspt2}
There are constants $r_0, \delta_0, \Lambda,\Lambda_2>0$ such that
\begin{align}\label{AS4}
|\{z:\kappa(z)> \Lambda\}\cap B_r|\geq \delta_0 r^d, \quad r\in (0,r_0]
\tag{\bf H$_4$};
\end{align}
\begin{align}\label{AS5}
\kappa(z)\leq \Lambda_2, \ z\in \R^d;\ \ \1_{\alpha=1}\int_{\{r<|z|<R\}} z\cdot\kappa(z)\dif z=0,\ 0<r<R<\infty
\tag{\bf H$_5$}.
\end{align}
\ea

\medskip

Recalling that $\cC:=\{x\in\R^d: 1/2< |x|<3/2\}$. Define
$$
\sB=\Big\{u\in \sS(\R^d): \mathrm{supp} \ \hat{u}\in \cC\Big\},\ \  J(u)=\Big\{x\in \R^d: |u(x)|=\|u\|_\infty \Big\}. $$

We have the following important frequency localized maximum principle.
\bl\label{Le-Max1}
There exists a number $c=c(d, \a, r_0, \delta_0)>0$ such that  for any $\kappa$ satisfying \eqref{AS4}, the following maximal inequality holds:
\be\label{Eq-Positive}
\inf_{u\in\sB} \inf_{x\in J(u)} \Big\{\sgn(u(x))\cdot (-\sL^\a_\kappa u(x))\Big\}\geq c\cdot\Lambda\|u\|_\infty.
\ee
\el

The following simple lemma is needed in the proof of Lemma \ref{Le-Max1}.

\bl\label{Le-Identity}
Suppose $f$ is a real analytic function on $\R^d$,  if $f$ vanishes on a measurable subset of $\R^d$  whose Lebesgue measure is positive, then $f\equiv 0$ on $\R^d$.
\el
\bpf
We prove the lemma by induction. Let $m_k$ be the Lebesgue measure on $\R^k$.
\begin{itemize}
	\item If $d=1$, then $f$ is analytic with $f|_E\equiv 0$ and $m_1(E)>0$, which implies zero points of $f$ must have an  accumulation point on the line, by identity theorem, $f\equiv 0$.
	\item Assume the claim holds for $d-1$. If $m_d(E)>0$, then by Fubini theorem, there is a set $E_1\subset \R$ with $m_1(E_1)>0$, such that for any $x_1\in E_1$,
	$$
	m_{d-1} (E\cap \{x_1\}\times \R^{d-1})>0.
	$$
	By  induction hypothesis, for each $x_1\in E_1$, function $z\mapsto f(x_1, z)$  vanishes identically.  Since $m_1(E_1)>0$, we can find $x_1^n\in E_1$, $x_1^n\to a$. Now for each $z\in \R^{d-1}$, function $x_1\mapsto f(x_1, z)$ is real analytic, its zero points has an accumulation point $a$. By the conclusion for 1 dimensional case, we get $f(x_1, z)\equiv 0$.
\end{itemize}
\epf

\bpf[\bf Proof of Lemma \ref{Le-Max1}]
 Without loss of generality, we can assume $\Lambda=1$. Define
$$
\sA(r_0, \delta_0):= \left\{\kappa: \kappa \ \mbox{satisfies}\  \eqref{AS4}\ \mbox{with}\ \Lambda=1\right\},
$$
$$
c:= \inf_{\kappa\in \sA(r_0,\delta_0)}\inf_{u\in\sB} \inf_{x\in J(u)} \Big\{\sgn(u(x))\cdot (-\sL^\a_\kappa u(x))\Big\}/\|u\|_\infty.
$$
We emphasize that the constant $c$ only depends on $d, \a, r_0, \delta_0$. By the definition of $c$, there exists a sequence of smooth functions $w_n\in \sS(\R^d)$ satisfying $\mathrm{supp}\, \hat{w}_n\subset \cC$, $x_n\in J({w_n})$ and $\kappa_n(z)\in \sA(r_0,\delta_0)$ such that
$$w_n(x_n)=\max_{x\in \R^d} |w_n| =1,\ \ \lim_{n\to\infty}\big[-\sL^\a_{\kappa_n} w_n(x_n)\big]=c. $$
Let $u_n(x):=w_n(x_n+x)$, it's easy to see that  $u_n\in \sB$ and
\be\label{eq-Ln(un)}
u_n(0)=\max_{x\in\R^d} |u_n|(x)=1, \quad \lim_{n\to\infty}\big[-\sL^\a_{\kappa_n} u_n(0)\big]=c.
\ee
Notice that
$$
u_n(x)=\int_{\R^d} \tilde{h}(x-y) u_n(y)\d y,
$$
where $\tilde{h}$ is defined in section 2. For any $k\in \N$,
$$
\|\nabla^k u_n\|_\infty=\|\nabla^k \tilde{h}* u_n\|_\infty\leq \|\nabla^k \tilde{h}\|_{1}\|u_n\|_\infty\leq C_k.
$$
By Ascoli-Azela's lemma and  diagonal argument, there is a subsequence of $\{u_n\}$(still denoted by $u_n$ for simple) and $u\in C_b^\infty$ such that $\nabla^k u_n$ converges to $\nabla^k u$ uniformly on any compact set. Let $\chi_R(\cdot)=\chi(\cdot/R)$, where $\chi$ is the same function in section 2. For any $\phi\in \sS(\R^d)$,
\begin{align*}
\left|\int \phi(u_n-u)\right|\leq& \int |\phi\chi_R\cdot (u_n-u) | + \int |\phi(1-\chi_R) (u_n-u) |\\
\leq& \|\phi\|_{L^1} \|u_n-u\|_{L^\infty(B_{3R/2})} + 2 \sup_{|x|>R} |\phi(x)|.
\end{align*}
Let $n\to \infty$ and then $R\to\infty$, we get
$$
\<\phi, u_n\>\to \<\phi, u\>, \quad \forall \phi\in \sS(\R^d).
$$
i.e.  $u_n\to u$ in $\sS'(\R^d)$ and consequently, $\hat{u}_n\to \hat{u}$ in $\sS'(\R^d)$. For any $\phi\in \sS(\R^d)$ supported on $\R^d\backslash\cC$, we have
$$
\<\phi, \hat{u}\>=\lim_{n\to\infty} \<\phi, \hat{u}_n\>=0,
$$
which means $u$ is also supported on $\cC$.  Thus the complex-valued function
$$
U: z\mapsto (2\pi)^{-d}  \<\e^{i\<z,\xi\>}, \hat{u}\>
$$
is a holomorphic function on $\mC^d$ and $u=U|_{\R^d}$. This implies $u$ is a real analytic function.

\medskip

Now assume $c=0$, for any $\lambda\in (0,1)$, by \eqref{eq-Ln(un)} and the fact that $\nabla u_n(0)=0$, we have
\begin{align*}
-\sL^\a_{\kappa_n} u_n(0) = &\int_{B_{r_0}}  (u_n(0)-u_n(z)) \frac{\kappa_n(z)}{|z|^{d+\a}} \d z\\
\geq& \int_{B_{r_0}\cap\{u_n\leq \lambda\}\cap \{\kappa_n > 1\}}  (1-u_n(z)) \frac{\d z}{|r_0|^{d+\a}}\\
\geq & (1-\lambda) r_0^{-d-\a}|B_{r_0}\cap\{u_n\leq \lambda\}\cap \{\kappa_n > 1\}|.
\end{align*}
This yields
$$
\limsup_{n\to\infty} |B_{r_0}\cap\{u_n\leq \lambda\}\cap \{\kappa_n > 1\}| \leq (1-\lambda)^{-1}r_0^{d+\a} \lim_{n\to\infty}[-\sL_{\kappa_n}^\a u_n(0)]=0.
$$
Combining the above estimate and our assumption \eqref{AS4}, we get
\begin{align*}
&\liminf_{n\to\infty}|B_{r_0}\cap\{u_n>\lambda\}\cap \{\kappa_n > 1\}| \\
=&\liminf_{n\to\infty}|B_{r_0}\cap \{\kappa_n > 1\}|-\limsup_{n\to\infty}|B_{r_0}\cap\{u_n\leq \lambda\}\cap \{\kappa_n > 1\}|\geq  \delta_0r_0^{d}.
\end{align*}
One the other hand, $u_n \to u$ uniformly in $B_{r_0}$ implies
\begin{align*}
&|B_{r_0}\cap\{u>\lambda\}| = \lim_{n\to\infty} |B_{r_0}\cap\{ u_n> \lambda\}|\\
\geq &\liminf_{n\to\infty}|B_{r_0}\cap\{u_n>\lambda\}\cap \{\kappa_n > 1\}| \geq \delta_0 r_0^d.
\end{align*}
Notice that $u\leq1$, let $\lambda\uparrow1$ in the first term above, we obtain $|\{x\in B_{r_0}: u(x)=1\}|\geq \delta_0r_0^d>0$. Using Lemma \ref{Le-Identity}, we obtain $u\equiv 1$ on $\R^d$ i.e. $\hat{u}=\delta_0$, the Dirac measure.  However, as we see before, $\hat{u}$ must be supported on $\cC$,
this contradiction implies $c=c(d, \a, r_0, \delta_0)>0.$
\epf

\iffalse
\br
Our proof seems hard to extend to some more singular cases, for example if
$$
L^\a u(x)=\sum_{k=1}^d \delta_z^\a u (x) \1_{\{\hat{z}_k=0\}} \frac{\d z_k}{|z_k|^{1+\a}},
$$
here $\hat{z}_k=(z_1,\cdots, z_{k-1}, z_{k+1},\cdots, z_d)$. Here is the reason: consider the following function,
$$
u(x)=1-\Pi_{k=1}^d \sin^2(x_k),
$$
obviously, $u|_S\equiv 1$, where $S$ is the support of $\nu$. Basic calculation shows,
$$
|\p^\a u(x)|= \Pi_{k=1}^d \left|\p^{\a_k}\Big((\sin^2(x_k))\Big)\right|\leq 2^{|\a|}.
$$
So $$\liminf_{n\to\infty}\|\nabla^n u\|_\infty^{1/n}\leq (n^d 2^n)^{1/n}\leq 2. $$
By real Paley-Wiener theorem(see \cite{An-Je}),
$\hat{u}$ is a function with compact support.
\er
\fi

\bc\label{Cor-Scal}
Let $R\geq 1$. Suppose $\kappa$ satisfies \eqref{AS4} and $\mathrm{supp}\ \hat u\subset R \cC:=\{x: x/R\in \cC\}$, then there is a positive constant $c=c(d, \a, r_0, \delta_0)$ such that
\be
\inf_{x\in J(u)} \Big\{\sgn(u(x))\cdot (-\sL_\kappa^\a u(x))\Big\}\geq c\, \Lambda R^\a\|u\|_\infty,
\ee
where $J(u)=\{x: |u(x)|=\|u\|_\infty\}$.
\ec
\bpf
Suppose $x_0\in J(u)$, define $u^{x_0}_R(x):=u(x_0+x/R)$, $\kappa_R(z):= R^{\a}\kappa(z/R)$.
By our assumption on $u$, one can see that $\mathrm{supp}\ \widehat{u_R^{x_0}}\subset \cC$ and $\kappa_R$ satisfies \eqref{AS4} with constant $\Lambda$ replaced by $\Lambda R^\a$.  Notice that
\begin{align*}
\begin{aligned}
\sL^\a_{\kappa_R} u^{x_0}_R(0)=&\int_{\R^d} \left(u(x_0+z/R)-u(x_0)\right) \frac{\kappa(z/R)}{|z/R|^{d+\a}} \d (z/R) \\
=& \int_{\R^d} \left(u(x_0+z)-u(x_0)\right) \frac{\kappa(z)}{|z|^{d+\a}} \d z=\sL^\a_\kappa u(x_0),
\end{aligned}
\end{align*}
by Lemma \ref{Le-Max1}, we obtain that
\begin{align*}
-\sgn(u(x_0))\cdot\sL^\a_\kappa u(x_0)=&-\sgn(u^{x_0}_R(0))\cdot\sL_{\kappa_R}^\a u^{x_0}_R(0) \\
\geq & c\, \Lambda R^\a \|u_R^{x_0}\|_\infty=c\,  \Lambda R^\a\|u\|_\infty.
\end{align*}
So we complete our proof.
\epf

We need the following simple commutator estimate.
\bl\label{Le-Comm}
For any $j\geq -1$, $\beta\in(0,1)$,
\be
\|[\Delta_j, b\cdot \nabla]u\|_\infty \leq C  2^{-\beta j} \|b\|_{\sC^\beta} \|\nabla u\|_{L^\infty},
\ee\label{Eq-Comm}
where $C=C(d, \beta)$.
\el
\begin{proof}
By \eqref{Eq-Delta-f}, we have
\begin{align*}
[\Delta_j, b\cdot \nabla]u(x)=\int_{\mR^d}h_j(y)(b(x-y)-b(x))\cdot \nabla u(x-y)\dif y,
\end{align*}
hence for any $\beta\in (0,1)$,
\begin{align}
\begin{split}
\|[\Delta_j, b\cdot\nabla]u\|_\infty&\leq \int_{\mR^d}h_j(y)\|b(\cdot-y)-b(\cdot)\|_{\infty}\|\nabla u\|_{L^\infty}\dif y\\
&\leq C \|b\|_{\sC^{\beta}}\|\nabla u\|_{L^\infty}\int_{\mR^d}|h_j(y)|\,|y|^{\beta}\dif y\\
&=C \|b\|_{\sC^\beta} \|\nabla u\|_{L^\infty} 2^{-\beta j} \int_{\mR^d}|2h(2y)-h(y)|\,|y|^{\beta}\dif y\\
&\leq C 2^{-\beta j} \|b\|_{\sC^\beta}\|\nabla u\|_{L^\infty}.
\end{split}
\end{align}
\end{proof}

\bt\label{Th-Holde-Const}
Suppose $\kappa$ satisfies \eqref{AS4},  if $\a\in (0,1]$ and $b\in \sC^\beta$ with $\beta\in (1-\a, 1)$; or $\a\in(1,2)$ and $b\in \sC^\beta$ with $\beta\in (-\frac{\a-1}{2}, 1)$. Then there a constant $\lambda_0$ such that for any $\lambda\geq \lambda_0$ and $f\in \sC^\beta$ \eqref{Eq-PDE} has a unique solution in $\sC^{\a+\beta}$. Moreover, we have the following apriori estimate
\be\label{Eq-Holder-Const}
(\lambda-\lambda_0)\|u\|_{\sC^\beta}+\|u\|_{\sC^{\a+\beta}}\leq C \|f\|_{\sC^\beta},
\ee
here $C=C(d,\a, \beta, r_0, \delta_0, \Lambda, \|b\|_{\sC^\beta})>0$, $\lambda_0=\lambda_0(d,\a, \beta, r_0, \delta_0, \Lambda, \|b\|_{\sC^\beta})\geq 0.$
\et
\bpf
 For $\a\in (0,1]$, we frist assume $u\in \sS(\R^d)$. Notice that $\Delta_j \sL_\kappa^\a=\sL^\a_\kappa\Delta_j$, we have
$$
\lambda \Delta_j u-\sL^\a_\kappa\Delta_j u-b\cdot\nabla \Delta_ju=\Delta_j f+[\Delta_j, b\cdot\nabla] u.
$$
For $j=-1$, then
$$
\lambda \Delta_{-1} u-\sL^\a_\kappa\Delta_{-1} u-b\cdot\nabla \Delta_{-1}u=\Delta_{-1} f+[\Delta_{-1}, b\cdot\nabla] u.
$$
Suppose $\Delta_{-1}u(x_{-1})=\|\Delta_{-1}u\|_\infty$, noticing $\sL^\a_\kappa\Delta_{-1} u(x_{-1})\leq 0$ and $\nabla u(x_{-1})=0$, we get
\begin{align*}
\lambda\|\Delta_{-1}u\|_\infty\leq& \lambda \Delta_{-1} u(x_{-1})-\sL^\a_\kappa\Delta_{-1} u(x_{-1}) \\
\leq &\|\Delta_{-1} f\|_\infty+\|[\Delta_{-1}, b\cdot\nabla] u\|_\infty\\
\leq&\|\Delta_{-1} f\|_\infty+C\|b\|_{\sC^\beta} \|u\|_{B^{1}_{\infty,1}}.
\end{align*}
For $j\geq 0$, assume $\sgn (\Delta_j u(x_j))\cdot\Delta_j u(x_j)=\|\Delta_ju\|_\infty$, by Lemma \ref{Le-Max1}
\begin{align*}
(\lambda+c 2^{\a j}))\|\Delta_j u\|_\infty=&\sgn (\Delta_j u(x_j))\cdot[\lambda \Delta_j u(x_j)+c2^{\a j} \Delta_j u(x_j)]\\
\leq& \|\lambda \Delta_j u-\sL^\a_\kappa\Delta_j u-b\cdot\nabla \Delta_ju\|_\infty\\
\leq & \|\Delta_jf\|_\infty+ \|[\Delta_j, b\cdot\nabla] u\|_\infty\\
\leq& \|\Delta_jf\|_\infty+  C 2^{-\beta j} \|b\|_{\sC^\beta} \|u\|_{B^1_{\infty,1}}.
\end{align*}
Combining the above inequalities and using interpolation,
\begin{equation*}
\begin{split}
\big(\lambda2^{\beta j} + c2^{(\a+\beta) j}\big)\|\Delta_j u\|_\infty \leq  2^{\beta j} \|\Delta_j f\|_\infty+ \|b\|_{\sC^\beta}\big(\eps \|u\|_{\sC^{\a+\beta}}+C_\eps \|u\|_{\sC^\beta}\big),
\end{split}
\end{equation*}
hence,
$$
(\lambda-C_\eps) \|u\|_{\sC^\beta}+(c-\eps C\|b\|_{\sC^\beta})\|u\|_{\sC^{\a+\beta}} \leq \|f\|_{\sC^\beta}.
$$
Choosing $\eps_0$ sufficiently small, such that $(c-\eps_0 \|b\|_{{\sC^\beta}})\geq \tfrac{c}{2}$, letting $\lambda_0=C_{\eps_0}$, we get \eqref{Eq-Holder-Const} for $u\in \sS(\R^d)$. Now if $u\in \sC^{\a+\beta}$, let $u_n:=n^d\eta(n\cdot)*\big(\chi(\tfrac{\cdot}{n})u\big)\in \sS(\R^d)$, where $\chi$ is the same function in section 2 and $\eta\in C_c^\infty(B_1)$, $\int \eta =1$. $f_n:=\lambda u_n-\sL^\a_\kappa u_n-b\cdot \nabla u_n$. So
 $$
(\lambda-\lambda_0)\|u_n\|_{\sC^\beta}+\|u_n\|_{\sC^{\a+\beta}}\leq C \|f_n\|_{\sC^\beta},
$$
by this, we obtain
\begin{align*}
(\lambda-\lambda_0)\|u\|_{\sC^\beta}+\|u\|_{\sC^{\a+\beta}}\leq &\limsup_{n\to\infty} \Big[(\lambda-\lambda_0)\|u_n\|_{{\sC^\beta}}+\|u_n\|_{{\sC^{\a+\beta}}}\Big]\\
\leq &C \limsup_{n\to\infty} \|f_n\|_{{\sC^\beta}}\leq C \|f\|_{{\sC^\beta}}.
\end{align*}

For $\a\in (1,2)$, we only prove the case $\beta\leq 0$ here. By choosing $\gamma\in (-\beta, \frac{\a-1}{2})$, and Bony's decomposition, we have
\begin{align*}
&\|\Delta_j(b\cdot \nabla u)\|_\infty\\
=&\left\|\Delta_j (T_b \nabla u)+\Delta_j (T_{\nabla u} b)+\Delta_j (R(b, \nabla u))\right\|_\infty\\
\leq & \sum_{\substack{k\leq l-2;\\|j-l|\leq 3}} \|\Delta_k b\|_\infty \|\Delta_l \nabla u\|_\infty+\sum_{\substack{l\leq k-2;\\|j-k|\leq 3}} \|\Delta_k b\|_\infty \|\Delta_l \nabla u\|_\infty+ \sum_{\substack{|k-l|\leq 1;\\k,l\geq j-2}}  \|\Delta_k b\|_\infty \|\Delta_l \nabla u\|_\infty\\
\leq & C_\gamma \|\nabla u\|_{\sC^\gamma} \|b\|_{\sC^\beta}\left(  j 2^{-\beta j}\cdot  2^{-\gamma j} +  2^{-\beta j}  + 2^{-(\beta+\gamma)j}\right)\\
\leq &C_\gamma \|u\|_{\sC^{1+\gamma}} \|b\|_{\sC^\beta} 2^{-\beta j}.
\end{align*}
Notice that,
\begin{align*}
\lambda \Delta_j u-\sL^\a_\kappa \Delta_j u =-\Delta_j (b\cdot \nabla u)+ \Delta_j f.
\end{align*}
Like before, we have
\begin{align*}
\begin{aligned}
(\lambda+c 2^{\a j}))\|\Delta_j u\|_\infty=&\sgn (\Delta_j u(x_j))\cdot[\lambda \Delta_j u(x_j)+c2^{\a j} \Delta_j u(x_j)]\\
\leq& \|\lambda \Delta_j u-\sL^\a_\kappa\Delta_j u\|_\infty\\
\leq & \|\Delta_jf\|_\infty+ \|\Delta_j(b\cdot\nabla u)\|_\infty\\
\leq& C 2^{-\beta j} (\|f\|_{\sC^\beta}+ \|b\|_{\sC^\beta} \|u\|_{\sC^{1+\gamma}}).
\end{aligned}
\end{align*}
Noticing that $1+\gamma<\a+\beta$, by interpolation, we get \eqref{Eq-Holder-Const}.
\epf

The next lemma will be used later.
\bl\label{Le-Lf-bdd}
Suppose $\kappa(z)$  satisfies \eqref{AS5}, then there is a constant $C=C(\alpha,d)>0$ such that for all $\beta\in\mR$ and $u\in \sC^{\a+\beta}$,
$$
\|\sL^{\alpha}_\kappa u\|_{\sC^\beta}\leq C\Lambda_2\|f\|_{\sC^{\a+\beta}}.
$$
\el
\begin{proof}
Recall that $\tilde{\varphi}$ is a smooth function supported in $B_2\setminus B_{1/4}$ with $\tilde {\varphi}=1$ on $B_{3/2}\setminus B_{1/2}$ and $\tilde h:=\cF^{-1}(\tilde {\varphi})$. Since $\tilde h\in\sS$, it is easy to see that for some $c=c(\alpha,d)>0$,
$$
\|\sL^\a_\kappa \tilde h\|_1\leq C\Lambda_2<\infty.
$$
Let $\tilde{h}_j:=\cF^{-1} (\tilde {\varphi}(2^{-j}\cdot))$ for $j=0,1,2,\cdots$. By scaling, we have
$$
\|\sL^\a_\kappa \tilde{h}_j\|_1\leq C\Lambda_2 2^{\a j},\ \ j=0,1,2,\cdots.
$$
Since $\widehat{\Delta_j f}=\varphi(2^{-j}\cdot) \hat f=\tilde {\varphi}(2^{-j}\cdot)\varphi(2^{-j}\cdot)\hat f$, we have $\Delta_j f=\tilde h_j*\Delta_j f$ and
$$
\|\Delta_j \sL^\a_\kappa  f\|_\infty= \|\sL^\a_\kappa (\tilde h_j* (\Delta_j f))\|_\infty
\leq\|\sL^\a_\kappa \tilde h_j\|_1\|\Delta_j f\|_\infty \leq C\Lambda_2 2^{\a j} \|\Delta_j f\|_\infty.
$$
Similarly, one can show
$$
\|\Delta_{-1} \sL^\a_\kappa  f\|_\infty\leq C\Lambda_2 \|\Delta_{-1}f\|_\infty.
$$
 Hence,
\begin{align*}
\|\sL^\a_\kappa f\|_{\sC^\beta}= \sup_{j\geq -1}2^{\beta j}\|\Delta_j \sL^\a_\kappa  f\|_\infty
\leq C\Lambda_2  \sup_{j\geq -1} 2^{\beta j} 2^{\a j} \|\Delta_j f\|_\infty=C\Lambda_2 \|f\|_{\sC^{\alpha+\beta}}.
\end{align*}
The proof is complete.
\end{proof}

\subsection{General case}
Denote
\begin{align*}
\delta_z f(x):= f(x+z)-f(x), \quad \delta_z^{\a}f(x):=& f(x+z)-f(x)-z^{(\a)}\cdot \nabla f(x).
\end{align*}
We need the following lemma.
\bl\label{Le-Lf-Var}
Suppose $\a\in (0,2)$ and $ \kappa(x,z) $ satisfies \eqref{AS2} and \eqref{AS3}, then
\begin{enumerate}
\item for any $\beta\in (0,\vartheta]$, we have
\be\label{Eq-Lf-bdd2}
\|\sL^\a_\kappa u\|_{\sC^\beta}\leq C \Lambda_2 \|u\|_{\sC^{\a+\beta}}+ C_\theta\Lambda_3 \|u\|_{\sC^{\a+\theta}},
\ee
where $\theta\in (0,\beta)$.
\item for any $\beta \in(-(\a\wedge \vartheta),0]$, we have
\be\label{Eq-Lf-bdd1}
\|\sL^\a_\kappa u\|_{\sC^\beta}\leq C (\Lambda_2+\Lambda_3) \|u\|_{\sC^{\a+\beta}}.
\ee
\end{enumerate}
\el

\bpf
(1). Suppose $\a\in (0,1]$ and $\beta\in(0, \vartheta]$.
For any $x_0\in \R^d$, define
$$
\sL^\a_{0}u(x)= \int_{\R^d} \delta_z^\a  u(x) \frac{\kappa(x_0,z)}{|z|^{d+\a}}\d z.
$$
Notice that $|\sL^\a_\kappa u(x_0)|=|\sL^\a_0 u(x_0)|,$ by Lemma \ref{Le-Lf-bdd}, we get
$$
\|\sL^\a_0 u\|_{L^\infty} \leq C\Lambda_2 \|u\|_{\sC^{\a+\beta}}.
$$
	
For any $x\in B_{1}(x_0)$ and $\theta\in (0,\beta)$, by definition
\begin{align*}
|\sL^\a_\kappa u(x)-\sL^\a _0u(x) |\leq& \left|\int_{\R^d} \delta_z^\a   u(x)\frac{(\kappa(x,z)-\kappa(x_0,z))}{|z|^{d+\a}}\d z\right|\\
\leq & \Lambda_3  |x-x_0|^\beta \int_{\R^d} |\delta_z^\a   u(x)|\frac{\d z}{|z|^{d+\a}}\\
\leq & C_\theta \Lambda_3|x-x_0|^\beta\|u\|_{C^{\a+\theta}}.
\end{align*} 	
Since
$$
|\sL^\a_\kappa u(x)-\sL^\a_\kappa u(x_0) |\leq |\sL^\a_\kappa u(x)-\sL _0^\a u(x) |+ |\sL^\a_0 u(x)-\sL^\a_0 u(x_0) |,
$$
by the Lemma \ref{Le-Lf-bdd}, if $\beta\in (0, \vartheta]$,
\begin{align*}
|\sL^\a_0 u(x)-\sL^\a_\kappa u(x_0) |\leq |\sL^\a_0 u(x)-\sL^\a_0 u(x_0) |\leq C \Lambda_2 \|u\|_{\sC^{\a+\beta}} |x-x_0|^\beta.
\end{align*}	
Combining the above inequalities, we get \eqref{Eq-Lf-bdd2}.

\medskip
(2).
We only prove the case $\a\in (1,2)$ and $\beta\in (-\vartheta, 0]$, which is harder and the only case that will be used below. Denote $\kappa_z(y):= \kappa(y,z),$
by definite we have
\be\label{eq-Ij}
\begin{aligned}
\Delta_j \sL^\a_\kappa u(x)=& \int_{\R^d} h_j(x-y)\  \d y \int_{\R^d} \delta_z^\a u(y) \ \frac{\kappa(y,z)}{|z|^{d+\a}}\  \d z \\
=& \int_{\R^d}  \left (\int_{\R^d} \big[\delta_z^\a u(y) \kappa_z(y)\big] h_j(x-y)\  \d y\right) \ \frac{\d z}{|z|^{d+\a}}.
\end{aligned}
\ee
Denote
$$
I_j(x, z)= \int_{\R^d} \big[\delta_z^\a u(y) \kappa_z(y)\big] h_j(x-y)\  \d y,
$$
We drop the index $x$ below for simple. By Bony's decomposition,
\begin{align*}
|I_j(z)| &=\left|\Delta_j \sum_{k, l\geq -1} \big[(\delta_z^\a \Delta_k  u)\cdot \Delta_l \kappa_z \big] \right| \\
=& \left|\Delta_j \left(\sum_{k\leq l-2} \delta_z^\a \Delta_k  u\cdot \Delta_l \kappa_z +\sum_{l\leq k-2} \delta_z^\a \Delta_k  u\cdot \Delta_l \kappa_z +\sum_{|k-l|\leq 1} \delta_z^\a \Delta_k  u\cdot \Delta_l \kappa_z \right)\right|\\
&\leq  \sum_{\substack{k\leq l-2;\\ |l-j|\leq 3}} \big|\delta_z^\a \Delta_k  u\cdot \Delta_l \kappa_z \big|+ \sum_{\substack{l\leq k-2;\\ |k-j|\leq 3}} \big| \delta_z^\a \Delta_k  u\cdot \Delta_l \kappa_z \big| +\sum_{\substack{k,l\geq j-2;\\|k-l|\leq 1}} \big|\delta_z^\a \Delta_k  u\cdot \Delta_l \kappa_z \big|\\
&=: I^{(1)}_j(z)+I^{(2)}_j(z)+I^{(3)}_j(z).
\end{align*}
Roughly speaking, the first inequality above holds because the Fourier transforms of $\sum_{k: k\leq l-2} \Delta_k f\Delta_l g$ and $\1_{|k-l|\leq1} \Delta_k f\Delta_l g$ are supported around $2^l\cC$ and $2^{l}B_1$ respectively. Noticing that by Bernstein's inequality
\be\label{eq-1order}
\begin{aligned}
|\delta_z^\a \Delta_k u(y)|=& \left| \int_0^1 z\cdot \big[\nabla \Delta_k u(y+tz) -\nabla \Delta_k u(y)\big] \d t\right|\\
\leq & 2|z| \|\nabla \Delta_k u\|_\infty \leq C \|u\|_{\sC^{\gamma}}  |z|2^{(1-\gamma)k},
\end{aligned}
\ee
and
\begin{align}\label{eq-2order}
|\delta_z^\a \Delta_k u(y)|\leq |z|^2 \|\nabla^2 \Delta_k u\|_\infty\leq C\|u\|_{\sC^{\gamma}} |z|^22^{(2-\gamma) k},
\end{align}
where $\gamma:= \a+\beta$.  Next we estimate each $I_j^{(i)}(z)$, we only need to care about the case when $j$ is large, say $j\geq 10$.
\begin{itemize}
\item If $|z|<2^{-j}$:  for $I_j^{(1)}(z)$, by \eqref{eq-2order} and noticing that $2-\gamma>0$,  we have
\be\label{eq-Ij1-1}
\begin{aligned}
I_j^{(1)}(z) =&\sum_{\substack{k\leq l-2; \\ |l-j|\leq 3}} \big|\delta_z^\a \Delta_k  u\cdot \Delta_l \kappa_z \big| \\
\leq& C \|u\|_{\sC^\gamma} \|\kappa_z\|_{\sC^\vartheta} \sum_{k\lesssim  j}  |z|^2 2^{(2-\gamma)k} 2^{-\vartheta j}\\
\leq&  C \|u\|_{\sC^\gamma} \|\kappa_z\|_{\sC^\vartheta} |z|^2 2^{(2-\gamma-\vartheta)j}.
\end{aligned}
\ee

Similarly, for $I_j^{(2)}(z)$, by \eqref{eq-2order} and noticing that $\vartheta>0$,  we have
\be
\begin{aligned}
I_j^{(2)}(z) \leq& C \|u\|_{\sC^\gamma}\|\kappa_z\|_{\sC^{\vartheta}} \sum_{l\lesssim j} |z|^2 2^{(2-\gamma)j} 2^{-\vartheta l}\\
\leq & C \|u\|_{\sC^\gamma}\|\kappa_z\|_{\sC^{\vartheta}} |z|^2 2^{(2-\gamma)j}.
\end{aligned}
\ee
\iffalse
Denote $h^x_j(y)= h_j(x-y)$. For $I_j^{(2)}(z)$, by integration by parts and noticing that $\vartheta\in (0,1)$, we have
\be\label{eq-Ij2-1}
\begin{aligned}
I_j^{(2)}(z) =& \sum_{\substack{l<k-2;\\ |k-j|\leq 2}}  \left|z\cdot \int_0^1 \d t\int_{\R^d} \big[\nabla \Delta_k u(y+tz)-\nabla\Delta_k u(y)\big] \Delta_l\kappa_z(y)h^x_j(y)\d y\right|\\
=&  \sum_{\substack{l<k-2;\\ |k-j|\leq 2}}  \left|z\cdot \int_0^1 \d t\int_{\R^d} \big[ \Delta_k u(y+tz)-\Delta_k u(y)\big] \nabla\big[\Delta_l \kappa_z(y)h^x_j(y)\big]\ \d y\right|\\
\leq & \sum_{\substack{l<k-2;\\ |k-j|\leq 2}} |z|^2 \|\nabla \Delta_k u\|_\infty \int_{\R^d} \left[|\nabla \Delta_l \kappa_l(y) \ h^x_j(y)|+  | \Delta_l \kappa_l(y) \ \nabla h^x_j(y)| \right]\d y\\
\leq &C \|u\|_{\sC^\gamma} \|\kappa_z\|_{\sC^\vartheta} \sum_{\substack{l<k-2;\\ |k-j|\leq 2}} |z|^2 2^{(1-\gamma)k} (2^{(1-\vartheta)l} +2^j 2^{-\vartheta l})\\
\leq & C \|u\|_{\sC^\gamma} \|\kappa_z\|_{\sC^\vartheta} |z|^2 2^{(2-\gamma-\vartheta) j} \end{aligned}
\ee
\fi
For $I^{(3)}_j(z)$,  we choose  $\eps_0\in (0, (\beta+\vartheta)\wedge (2-\a))$, by \eqref{eq-1order},
\eqref{eq-2order} and noticing that $1-\gamma-\vartheta<0$ and $2-\a-\eps_0>0\vee(2-\gamma-\vartheta)$, we have
\be
\begin{aligned}
&I^{(3)}_j(z) \leq \sum_{\substack{|k-l|\leq 1; \\k, l> -\log_2 |z|-2}}\big|\delta_z^\a \Delta_k  u\cdot \Delta_l \kappa_z \big| + \sum_{\substack{|k-l|\leq 1;\\j-3\leq k, l \leq -\log_2 |z|}} \big|\delta_z^\a \Delta_k  u\cdot \Delta_l \kappa_z \big|\\
\leq &C\|\kappa_z\|_{\sC^\vartheta} \|u\|_{\sC^\gamma}\Big( \sum_{k\geq -\log_2 |z|}|z|2^{(1-\gamma)k} 2^{-\vartheta k}+\sum_{j-2\leq k\leq -\log_2|z|}  |z|^2 2^{(2-\gamma)k} 2^{-\vartheta k}\Big) \\
\leq & C\|u\|_{\sC^\gamma} \|\kappa_z\|_{\sC^\vartheta} \left(|z|^{\gamma+\vartheta}+ |z|^2 \sum_{k\leq -\log_2 |z|} 2^{(2-\a-\eps_0)k} \right)\\
\leq & C \|u\|_{\sC^\gamma} \|\kappa_z\|_{\sC^\vartheta} |z|^{\a+\eps_0}.
\end{aligned}
\ee
\item $|z|\geq 2^{-j}$: for $I_j^{(1)}(z)$, notice that $\gamma<2$ and $1-\gamma-\vartheta<0$, we have
\be\label{eq-Ij1-2}
\begin{aligned}
 &I_j^{(1)}(z)\leq C\|u\|_{\sC^\gamma} \|\kappa_z\|_{\sC^\vartheta} \left( \sum_{-(1\wedge \log_2 |z|)\leq k\leq j} |z|2^{(1-\gamma)k} 2^{-\vartheta j} +\sum_{-1\leq k< -(1\wedge \log_2|z|)} |z|^2 2^{(2-\gamma)k} 2^{-\vartheta j}\right)\\
 \leq& C\|u\|_{\sC^\gamma} \|\kappa_z\|_{\sC^\vartheta} \left(\1_{\gamma>1}|z|^{\gamma} 2^{-\vartheta j}+ \1_{\gamma=1} |z| j2^{-\vartheta j}+\1_{\gamma<1} |z|2^{(1-\gamma-\vartheta)j} + |z|^\gamma 2^{-\vartheta j}\right)\\
 \leq &C \|u\|_{\sC^\gamma} \|\kappa_z\|_{\sC^\vartheta} (|z|^{\gamma}+\1_{\gamma<1} |z|2^{(1-\gamma-\vartheta)j} ).
 \end{aligned}
\ee
For $I_j^{(2)}(z)$, by \eqref{eq-1order} noticing that $\vartheta>0$, we have
\be\label{eq-Ij2-2}
\begin{aligned}
I_j^{(2)}(z)\leq&  \sum_{\substack{l\leq k-2;\\ |k-j|\leq 3}} \big| (\delta_z^\a \Delta_k  u\cdot \Delta_l \kappa_z) \big| \\
\leq& C\|u\|_{\sC^\gamma} \|\kappa_z\|_{\sC^\vartheta} \sum_{l\lesssim j} |z|2^{(1-\gamma)j} 2^{-\vartheta l} \\
\leq&  C\|u\|_{\sC^\gamma} \|\kappa_z\|_{\sC^\vartheta} |z|2^{(1-\gamma)j}.
\end{aligned}
\ee
For $I_j^{(3)}(z)$, by \eqref{eq-2order}, and notice that $1-\gamma-\vartheta<0$, we have
\be\label{eq-Ij3-2}
\begin{aligned}
I_j^{(3)}(z)\leq&   C\|u\|_{\sC^\gamma} \|\kappa_z\|_{\sC^\vartheta} \sum_{\substack{k,l\geq j-2;\\|k-l|\leq 2}}  |z| 2^{(1-\gamma)k} 2^{-\vartheta l}\\
\leq& C\|u\|_{\sC^\gamma} \|\kappa_z\|_{\sC^\vartheta} |z| 2^{(1-\gamma-\vartheta)j}.
\end{aligned}
\ee
\end{itemize}
Combining \eqref{eq-Ij1-1}-\eqref{eq-Ij3-2} and recalling that $\gamma=\a+\beta$, we obtain that for each $x\in \mR^d$,
\be\label{eq-Integral-Ij}
\begin{aligned}
|\Delta_j \sL_\kappa^\a u(x)| =&\left| \int_{\R^d} I_j(x, z) \frac{\d z}{|z|^{d+\a}}\right|\\
\leq &C\|u\|_{\sC^\gamma} \sup_{z\in \R^d}\|\kappa_z\|_{\sC^\vartheta}  \Big(2^{(2-\gamma)j}\int_{|z|<2^{-j}} |z|^{2-d-\a} \d z+\int_{|z|<2^{-j}}|z|^{\eps_0-d} \d z\\
&+ \int_{|z|\geq 2^{-j}} |z|^{\gamma-d-\a}\d z+ 2^{(1-\gamma)j} \int_{|z|\geq 2^{-j}} |z|^{1-d-\a} \d z\Big)\\
=& C\|u\|_{\sC^\gamma} \sup_{z\in \R^d}\|\kappa_z\|_{\sC^\vartheta}  \Big(2^{-\beta j}+1+ 2^{-\beta j}\Big)\\
\leq &  C\|u\|_{\sC^\gamma} \sup_{z\in \R^d} \|\kappa_z\|_{\sC^\vartheta} 2^{-\beta j}.
\end{aligned}
\ee
i.e.
$$
\|\sL_\kappa^\a u\|_{\sC^\beta} = \sup_{j\geq -1} 2^{-\beta j}\|\Delta_j \sL^\a_\kappa u\|_\infty\leq C\|u\|_{\sC^{\a+\beta}} \sup_{z\in \R^d} \|\kappa_z\|_{\sC^\vartheta}.
$$
So we complete our proof.
\epf

Before we proving our main results, let us give a brief discussion about our assumptions on $\kappa(x,z)$: let$\Lambda=\Lambda_1/(2c_d)$, where $c_d$ is the volume of unity ball in $\R^d$. By our assumptions \eqref{AS1} and \eqref{AS2}, we can see that for any $r\in (0,r_0]$, $x\in \R^d$,
\begin{align*}
& |B_r\cap \{\kappa(x, \cdot)\geq \Lambda\}| \geq \Lambda_2^{-1}\int_{B_r\cap \{\kappa(x,\cdot)\geq \Lambda\}} \kappa(x, z)\d z\\
=& \Lambda_2^{-1}\int_{B_r} \kappa(x, z)\d z-\Lambda_2^{-1}\int_{B_r\cap \{\kappa(x,\cdot)< \Lambda\}} \kappa(x, z)\d z\\
\geq& \Lambda_2^{-1}(\Lambda_1r^d-\Lambda |B_r|)\geq \frac{\Lambda_1}{2\Lambda_2} r^d.
\end{align*}
Thus, for each $x\in \R^d$, $\kappa(x,\cdot)$ satisfies \eqref{AS4} with $\Lambda=\Lambda_1/(2c_d)$ and $\delta_0=\Lambda_1/(2\Lambda_2)$.

Now we give the proof for Theorem \ref{Th-Holder-Var}.
\begin{proof}[\bf Proof of Theorem \ref{Th-Holder-Var}]
(1) Define
$$\sL^\a_0 u(x)= \int_{\R^d} \delta^\a_z  f(x) \frac{\kappa(x_0,z)}{|z|^{d+\a}}\d z.$$
Choose $\eta$ be a smooth function with compact support in $B_1$ and $\eta(x)=1$, if $x\in B_{\frac{1}{2}}$. Fixed $x_0\in \R^d$, define

$$\eta_\eps^{x_0}(x):=\eta\Big(\frac{x-x_0}{\eps}\Big);\,\,\, \kappa^{x_0}_\eps(x,z):=[\kappa(x,z)-\kappa(x_0,z)]\eta_\eps(x). $$
We omit the supscript $x_0$ below for simple. Define $v=u\eta_\eps$, then we have
\be\label{loc}
\begin{split}
&\lambda v-\sL^\a_0 v-b\cdot \nabla v\\
=&[\eta_\eps f-ub\cdot \nabla \eta_\eps+u\sL_\kappa^\a\eta_\eps]+\eta_\eps (\sL_\kappa^\a u-\sL^\a_0 u)+[\eta_\eps\sL^\a_0 u-\sL^\a_0(\eta_\eps u)-u\sL_\kappa^\a\eta_\eps].
\end{split}
\ee
Obviously,
\be\label{loc-1}
\|\eta_\eps f-ub\cdot \nabla \eta_\eps+u\sL_\kappa^\a\eta_\eps\|_{\sC^\beta} \leq C_\eps (\|f\|_{\sC^\beta}+\|u\|_{\sC^\beta} ).
\ee
Denote
$$
{\widetilde w}_\eps(x):=\eta_\eps(x) (\sL_\kappa^\a u(x)-\sL^\a_0 u(x))=\int_{\R^d} \delta^\a_z  u \frac{\kappa_\eps(x,z)}{|z|^{d+\a}}\d z.
$$
By \eqref{Eq-Lf-bdd2}, for any $\theta\in(0,\beta)$
\be\label{loc-2}
\begin{split}
\|\widetilde w_\eps\|_{\sC^\beta} \leq& C \sup_{z}\|\kappa_\eps(\cdot,z)\|_{L^\infty}\|u\|_{\sC^{\a+\beta}}+C_\theta \sup_{z}[\kappa_\eps(\cdot,z)]_{\vartheta}\|u\|_{\sC^{\a+\theta}}\\
\leq &C\eps^\beta  \|u\|_{\sC^{\a+\beta}}+ C_{\theta,\eps}\|u\|_{\sC^{\a+\theta}}.
\end{split}
\ee
Denote
$$
w_\eps(x):=[\eta_\eps\sL^\a_0 u-\sL^\a_0(\eta_\eps u)-u\sL^\a_0\eta_\eps](x)
$$
and $\delta_zf(x)=(f(x+z)-f(x))$, by definition, we have
\be\label{w-eps2}
 w_\eps(x)
=\int_{\R^d}\delta_z \eta_\eps(x)\ \delta_z u (x)\frac{\kappa(x_0, z)}{|z|^{d+\a}} \d z, \ee
and
\be\label{delta-w}
\begin{split}
w_\eps(x)-w_\eps(y)=& \int_{\R^d} \delta_z \eta_\eps (x)\big[\delta_z u (x)-\delta_z u (y)\big]\frac{\kappa(x_0, z)}{|z|^{d+\a}}\d z\\
&+ \int_{\R^d} \big[\delta_z \eta_\eps (x)-\delta_z \eta_\eps(y)\big]\delta_z u(y)\frac{\kappa(x_0, z)}{|z|^{d+\a}}\d z.
\end{split}
\ee
In order to estimate the $\sC^\beta$ norm of
$w_\eps$, for different cases we have to deal it separately.

(i)For $\a\in(0,1)$, by \eqref{w-eps2},
\begin{align*}
\big|w_\eps(x)|\leq& \int_{|z|\leq 1} \|\nabla \eta_\eps\|_{L^\infty}\|u\|_{L^\infty} |z|\frac{\kappa(x_0, z)}{|z|^{d+\a}}\d z+2\int_{|z|>1} \|\eta_\eps\|_{L^\infty} \|u\|_{L^\infty} \frac{\kappa(x_0, z)}{|z|^{d+\a}}\d z\\
\leq &C_\eps \|u\|_{L^\infty}.
\end{align*}
And by \eqref{delta-w},
\begin{equation*}
\begin{split}
&\big|w_\eps(x)-w_\eps(y)\big|\\
\leq&C |x-y|^\beta \Big(\|u\|_{C^\beta} \|\nabla \eta_\eps\|_{L^\infty} \int_{|z|\leq 1} \frac{\d z}{|z|^{d+\a-1}} +\|u\|_{C^\beta}  \|\eta_\eps\|_{L^\infty} \int_{|z|>1} \frac{\d z}{|z|^{d+\a}} \Big)\\
&+C  |x-y|^\beta \Big(\|\eta_\eps\|_{\sC^{1+\beta}} \|u\|_{L^\infty} \int_{|z|\leq 1} \frac{\d z}{|z|^{d+\a-1}}+  \|\eta_\eps\|_{C^\beta}  \|u\|_{L^\infty} \int_{|z|>1} \frac{\d z}{|z|^{d+\a}} \Big)\\
\leq &C_\eps |x-y|^\beta\|u\|_{\sC^\beta}.
\end{split}
\end{equation*}
Hence, we have
\be\label{loc-3} \|w_\eps\|_{ \sC^\beta}\leq C_\eps\|u\|_{ \sC^\beta}. \ee
Let $\lambda_0'$ be the constant $\lambda_0$ in Theorem \ref{Th-Holde-Const}, by \eqref{loc}, \eqref{loc-1}, \eqref{loc-2}, \eqref{loc-3}, Theorem \ref{Th-Holde-Const}, interpolation theorem and the discussion before this proof, we have
\begin{align*}
&\|u\|_{C^{\a+\beta} (B_{\eps/2}(x_0))} +(\lambda-\lambda'_0) \|u\eta_\eps^{x_0}\|_{\sC^{\beta}}\\
\leq & C \|v\|_{\sC^{\a+\beta}} +(\lambda-\lambda_0')\|v\|_{\sC^{\beta}}\\
\leq& C\eps^\beta \|u\|_{\sC^{\a+\beta}}+ C_{\theta,\eps} \|u\|_{\sC^{\a+\theta}}+C\|f\|_{\sC^\beta}\\
\leq& C\eps^\beta \|u\|_{C^{\a+\beta}}+ C_{\theta,\eps} \|u\|_{\sC^{\beta}}+C\|f\|_{\sC^\beta}\\
\leq & C\eps^\beta \sup_{x_0\in \R^d}\|u\|_{C^{\a+\beta}(B_{\eps/2}(x_0))}+ C_{\theta,\eps} \|u\|_{\sC^{\beta}}+C\|f\|_{\sC^\beta}.
\end{align*}
We can fixed $\eps_0$ sufficiently small, such that
$C\eps_0^\beta\leq 1/2$, so we have
\begin{align*}
\sup_{x_0\in\R^d} \left( \|u\|_{C^{\a+\beta} (B_{\eps_0/2}(x_0))} +(\lambda-\lambda'_0) \|u\eta_\eps^{x_0}\|_{\sC^{\beta}} \right)\leq C_{\eps_0}(\|f\|_{\sC^\beta}+\|u\|_{\sC^\beta}).
\end{align*}
This yields
$$
\|u\|_{\sC^{\a+\beta}} \leq C_{\eps_0} \sup_{x_0\in\R^d}\|u\|_{C^{\a+\beta}(B_{\eps_0/2}(x_0))}\leq  C_{\eps_0} \left(\|u\|_{\sC^\beta}+\|f\|_{\sC^\beta}\right),
$$
and
$$
C_{\eps_0}(\|f\|_{\sC^\beta}+\|u\|_{\sC^\beta})\geq (\lambda-\lambda_0') \sup_{x_0\in \R^d}\|u\eta_{\eps_0}^{x_0}\|_{\sC^\beta}\geq c_{\eps_0} (\lambda-\lambda_0') \|u\|_{\sC^\beta},
$$
where $c_{\eps_0}$ is a constant larger than $0$. Thus,
$$
\|u\|_{\sC^{\a+\beta}} + (\lambda-\lambda'_0) \|u\|_{\sC^\beta}\leq C_{\epsilon_0}(\|f\|_{\sC^\beta}+\|u\|_{\sC^\beta}).
$$
Letting $\lambda_0=\lambda_0'+C_{\eps_0}$, we obtain \eqref{Eq-Holder-Var}.

(ii)For $\a=1$, by \eqref{w-eps2} and \eqref{delta-w}, we have
$$
\|w_\eps\|_{L^\infty} \leq C_\eps\|u\|_{C^1},
$$
and
\begin{equation}\label{loc-4}
\begin{split}
&\left|w_\eps(x)-w_\eps(y)\right|\\
\leq& \left| \int_{|z|\leq \delta} \delta_z \eta_\eps (x) \Big[\delta_z u(x)-\delta_z u(y)\Big]\frac{\kappa(x_0,z)}{|z|^{d+\a}}\ \d z\right|+\left| \int_{|z|> \delta} \delta_z \eta_\eps (x) \Big[\delta_z u(x)-\delta_z u(y)\Big]\frac{\kappa(x_0,z)}{|z|^{d+\a}}\ \d z\right|\\
&+  \left|\int_{|z|\leq \delta} \delta_z u(y)\Big[\delta_z \eta_\eps (x)-\delta_z\eta_\eps(y)\Big]  \frac{\kappa(x_0,z)}{|z|^{d+\a}}\ \d z\right|+   \left|\int_{|z|> \delta} \delta_z u(y)\Big[\delta_z \eta_\eps (x)-\delta_z\eta_\eps(y)\Big]  \frac{\kappa(x_0,z)}{|z|^{d+\a}}\ \d z\right|\\
\leq& C\eps^{-1}  |x-y|^\beta \|u\|_{\sC^{1+\beta}} \int_{|z|\leq \delta}  \frac{\d z}{|z|^{d+\a-2}} +C  |x-y|^\beta \|u\|_{C^\beta} \int_{|z|>\delta} \frac{\d z}{|z|^{d+\a}}\\
&+C\eps^{-1-\beta}  |x-y|^\beta \|\nabla u\|_{L^\infty} \int_{|z|\leq \delta} \frac{\d z}{|z|^{d+\a-2}}\d z+C\eps^{-\beta}|x-y|^\beta \|u\|_{L^\infty}\int_{|z|>\delta} \frac{\d z}{|z|^{d+\a}}\\
\leq &\Big(C \eps^{-2} \delta^{2-\a} \|u\|_{\sC^{1+\beta}} +C(\eps,\delta)  \|u\|_{\sC^\beta}\Big)|x-y|^{\beta}.
\end{split}
\end{equation}
Hence,
$$ \|w_\eps\|_{\sC^{\beta}}\leq C\eps^{-2}\delta^{2-\a} \|u\|_{\sC^{1+\beta}}+C(\eps,\delta)  \|u\|_{\sC^\beta}. $$
Choosing $\delta=\eps^{\frac{2+\beta}{2-\a}}$, by Theorem \ref{Th-Holde-Const}, interpolation and above inequality, we get
\begin{align*}
&\|u\|_{\sC^{1+\beta} (B_{\eps/2}(x_0))} +(\lambda-\lambda'_0) \|u\|_{\sC^{\beta}(B_{\eps/2}(x_0))}\\
\leq& C\eps^\beta \|u\|_{\sC^{\a+\beta}}+ C_{\theta,\eps} \|u\|_{\sC^{1+\theta}}+C\eps^{-2}\delta^{2-\a} \|u\|_{\sC^{1+\beta}}+C(\eps,\delta)  \|u\|_{\sC^\beta}+C\|f\|_{\sC^\beta}\\
\leq &C \eps^\beta \|u\|_{\sC^{1+\beta}} +C_\eps  \|u\|_{\sC^\beta}+C\|f\|_{\sC^\beta}.
\end{align*}
Like the above case, we get \eqref{Eq-Holder-Var}.

\medskip

(2) For $\a\in (1,2)$, we only give the proof for $\beta\leq 0$ here. Like the previous cases, we have \eqref{loc}. Moreover, notice that $\beta\in (-(\frac{\a-1}{2}\wedge \vartheta), 0]$, it is easy to see that
$$
\|\eta_\eps f-ub\cdot \nabla \eta_\eps+u\sL_\kappa^\a\eta_\eps\|_{\sC^\beta} \leq C_\eps (\|f\|_{\sC^\beta}+\|u\|_{\sC^\vartheta} ),
$$
and
$$
\|w_\eps\|_{\sC^\beta}\leq C\|w_\eps\|_{L^\infty} \leq C \|u\|_{C^1}.
$$
For $\tilde w_\eps$, fixing $\gamma\in (-\beta, \vartheta)$, then for any $z\in \R^d$,
\begin{align*}
\| \kappa_\eps (\cdot,z)\|_{\sC^\gamma} =& \|[\kappa(\cdot,z)-\kappa(x_0,z)]\eta_\eps(\cdot)\|_{\sC^\gamma}\\
\leq & C (\eps^{\vartheta} [\eta_\eps]_{\gamma}+ [\kappa(\cdot,z)]_{C^\gamma(B_\eps(x_0))})\\
\leq & C \eps^{\vartheta-\gamma}.
\end{align*}
Using Lemma \ref{Le-Lf-Var} (2)(replace $\vartheta$ with $\gamma$) and above inequality, we obtain
\begin{align*}
\| {\tilde w}_\eps\|_{\sC^\beta}:=& \| \eta_\eps (\sL_\kappa^\a u-\sL^\a_0 u)\|_{\sC^\beta}=\left\| \int_{\R^d} \delta_z  u(\cdot) \frac{\kappa_\eps(\cdot,z)}{|z|^{d+\a}}\d z\right\|_{\sC^\beta} \\
\leq & C \eps^{\vartheta-\gamma} \|u\|_{\sC^{\a+\beta}}.
\end{align*}
Now by the similar argument as in the previous case, we get  \eqref{Eq-Holder-Var}.

\end{proof}

\section{Martingale Problem and Weak Solution}
Before going to the definition of martingale problem associated with $\sL^\a_{\kappa,b}$, let us briefly introduce the corresponding SDE.

Let $(\Omega, \bP, \cF)$ be a probability space and $N(\d r, \d z, \d s)$ be a Poisson random measure on $\R_+\times \R^d\times \R_+$ with intensity measure is $\d r\,\frac{\d z}{|z|^{d+\a}}\,\d s$.
Define
\begin{align}\label{Nalpha}
N^{(\a)}(\d r, \d z, \d s)=\left\{
\begin{aligned}
&N(\d r, \d z, \d s) &\quad \a\in(0,1)\\
&N(\d r, \d z, \d s)-\d r \frac{\1_{B_1}(z) \d z}{|z|^{d+\a}} \d s &\quad \a=1\\
&N(\d r, \d z, \d s)-\d r\frac{\d z}{|z|^{d+\a}}\d s &\quad \a\in(1,2)
\end{aligned}
\right.
\end{align}
Consider the following SDE driven by Poisson random measure $N$:
\be\label{Eq-Nsde}
X_t=X_0+\int_0^t \int_{\R^d}\int_0^\infty z\1_{ [0,\kappa(X_{s-},z))}(r) N^{(\a)}(\d r, \d z, \d s)+\int_0^t b(X_s) \d s.
\ee

As mentioned before, when $b$ is just a distribution,  the drift term $``\int_0^\cdot b(X_s) \d s"$ may not be a process with finite variation any more but an additive functional of $X$ with zero energy, which means $X$ may not be a semimartingale but a Dirichlet process. We give the precious definitions of Dirichlet processes and process of zero energy first.
\begin{definition}
 We say that a continuous adapted process $(A_t)_{t\in[0,T]}$ is a process of zero energy if $A_0=0$ and
 \begin{align*}
 \lim_{\delta\rightarrow0}\sup_{|\pi_T|<\delta}\bE\Big(\sum_{t_i\in\pi_T}|A_{t_{i+1}}-A_{t_i}|^2\Big)=0
 \end{align*}
 where $\pi_T$ denotes a finite partition of $[0,T]$ and $|\pi_T|$ denotes the mesh size of the partition.
 \end{definition}
 \begin{definition}\label{Diri}
 We say that an adapted process $(X_t)_{t\in[0,T]}$ is a Dirichlet process if
 \begin{align}\label{Diric}
 X_t=M_t+A_t
 \end{align}
 where $M$ is a square-integrable martingale and $A$ is an adapted process of zero energy.
\end{definition}
 \medskip

Suppose $\kappa(\cdot,z), b$ is smooth and bounded, then the above equation has a unique solution. By It\^o's formula(see \cite[Theorem 4.4.7]{Ap}), for any $f\in C_b^2$, we  have
\be\label{eq-Ito}
\begin{aligned}
&f(X_t)-f(X_0)\\
=& \int_0^t \int_{\R^d}\int_0^\infty [f(X_{s-}+z\1_{ [0,\kappa(X_{s-},z))}(r))-f(X_{s-})] \widetilde{N} (\d r, \d z, \d s)+ \int_0^t b\cdot\nabla f(X_s) \d s \\
&+\int_0^t \int_{\R^d}\int_0^\infty [f(X_{s-}+z\1_{ [0,\kappa(X_{s-},z))}(r))-f(X_{s-})-z\1_{ [0,\kappa(X_{s-},z))}(r)] \d r \frac{\d z}{|z|^{d+\a}} \d s \\
=&M^f_t+ \int_0^t \sL_{\kappa}^\a f(X_s)\d s+\int_0^t b\cdot\nabla f(X_s) \d s,
\end{aligned}
\ee
where
$$
M^f_t:=\int_0^t \int_{\R^d}\int_0^\infty [f(X_{s-}+z\1_{ [0,\kappa(X_{s-},z))}(r))-f(X_{s-})] \widetilde{N} (\d r, \d z, \d s).
$$

Thus, \eqref{Eq-Nsde} is the SDE associated with operator $\sL^\a_{\kappa, b}$ at least when the coefficients are regular. However, when $b\in \sC^\beta$ with $\beta\leq 0$, we must face up to the  problem of how to define the term $``\int_0^t b(X_{s})\dif s"$ in \eqref{Eq-Nsde} and $``\int_0^t b\cdot \nabla f(X_{s})\dif s"$ in \eqref{eq-Ito}.  Inspired by \cite{Zh-Zh1}, when considering the martingale problem associated with $\sL_{\kappa, b}^\a$, if $b\in \sC^\beta$ with $\beta\leq 0$, we need restrict ourselves to some probability measures on $\mD:=D(\R_+;\R^d)$ satisfying the following Krylov's type estimate:
\bd\label{Def-Kry}
(Krylov's type estimate) We call a probability measure $\mP\in\sP(\mD)$ satisfy Krylov's estimate with indices $\mu$ if for any $T>0$,
there are positive constants $C_T$ and $\gamma$
such that for all $f\in C^\infty$, $0\leq t_0<t_1\leq T$,
\begin{align}\label{Eq-Kr}
\mE\left|\int^{t_1}_{t_0}f(w_s)\dif s\right|^2\leq C_{T}|t_1-t_0|^{1+\gamma}\|f\|^2_{\sC^\mu},
\end{align}
where the expectation $\mE$ is taken with respect to $\mP$. All the probability measure $\mP$ with property \eqref{Eq-Kr} is denoted by $\sK^{\mu}(\mD)$.
\ed
We should point out that for arbitrary $f\in \sC^\beta$, there is no good smooth approximation sequence in space $\sC^\beta$. However,  the modifying approximation sequence $f_n:=f* \eta_n$ converges to $f$ in $\sC^\mu$, for any $\mu<\beta$. So given $f\in \sC^\beta$ with $\beta\leq 0$, in order to give a natural definition of $\int_0^t f(\om_s) \d s$ under some suitable  probability measure $\mP$, we have to restrict ourselves to $\mP\in \sK^\mu(\mD)$ with $\mu<\beta$.

\bp\label{Pr-Af}
Let $\mu< \beta\leq 0$, $\mP\in\sK^{\mu}(\mD)$, for any $f\in \sC^\beta$, there is a continuous $\cB_t(\mD)$-adapted process $A^f_t$ with zero energy and such that for any $T>0$,
\begin{align}\label{NL3}
\lim_{n\to\infty}\mE\left(\sup_{t\in[0,T]}\left|\int^t_0f_n(w_s)\dif s-A^f_t\right|\right)=0,
\end{align}
where $C_b^\infty\ni f_n%:=f*\eta_n
\overset{\sC^\mu}{\longrightarrow}f$. Moreover, the mapping $\sC^\mu \ni f\mapsto A^f_{\cdot}\in L^2(\mD,\mP; C([0,T]))$ is
a bounded linear operator and
for all $0\leq t_0<t_1\leq T$,
\begin{align}\label{Kr0}
\mE\left|A^f_{t_1}-A^f_{t_0}\right|^2\leq C_{T}(t_1-t_0)^{1+\gamma}\|f\|^2_{\sC^\mu},
\end{align}
where the constants $C_{T}$ and $\gamma$ are the same as in \eqref{Eq-Kr}.
\ep
Since the proof for this proposition is just the same with Proposition 3.2 in \cite{Zh-Zh1}, we omit the details here.

\medskip

Now we are on the position to give the definition of martingale problem.
\bd[Martingale Problem]\label{Def-MP}
\begin{enumerate}
\item If $b\in B_b(\R^d)$, we call a probability measure $\mP\in\sP(\mD)$ a martingale solution associated with $\sL^\a_{\kappa,b}$ starting from $x\in\mR^d$ if for any $f\in C_b^\infty$,
\begin{align}\label{Eq-MP1}
M^f_t:= f(\om_t)-f(x)-\int_0^t \sL^\a_{\kappa, b} f(\om_s)\dif s
\end{align}
is a continuous $\cB_t(\mD)$-martingale with $M^f_0=0$ under $\mP$. The set of the martingale solutions with starting point $x$ is denoted by $\sM_{\kappa,b}(x)$.
\item If $\mu<\beta\leq 0$, $b\in \sC^\beta$ with $\beta\leq 0$, we call a probability measure $\mP\in\sK^{\mu}(\mD)$ a martingale solution associated with $\sL^\a_{\kappa,b}$ starting from $x\in\mR^d$
if for any $f\in C_b^\infty$,
\begin{align}\label{Eq-MP2}
M^f_t:=f(w_{t})-f(x)-\int^{t}_0 \sL_\kappa^\a f(w_s)\dif s-A^{b\cdot\nabla f}_{t}
\end{align}
is a continuous $\cB_t(\mD)$-martingale with $M^f_0=0$ under $\mP$. The set of the martingale solutions $\mP\in\sK^{\mu}(\mD)$  and starting point $x$
is denoted by $\sM^{\mu}_{\kappa,b}(x)$.
\end{enumerate}
\ed

\medskip

By Theorem \ref{Th-Holder-Var} (1), immediately, we have
\bl\label{Le-WP1}
Suppose $\a\in (0,1]$, $\kappa(x,z)$ satisfies \eqref{AS1}-\eqref{AS3} with \\ $\max\{{0, (1-\a)}\}<\vartheta<1$, and $b\in \sC^\beta$ with $\beta\in (0, \vartheta)$,  then for any $x\in \R^d$, there is a unique element in $\sM_{\kappa, b}(x)$.
\el
\bpf
The Existence of martingale solution to \eqref{Eq-MP1} is trivial, since the coefficients are globally H\"older continuous. We only give the proof for uniqueness. Suppose $\mP_x \in \sM_{\kappa,b}(x)$. For any $f\in C_b^\infty$ and $\lambda\geq  \lambda_0$, where $\lambda_0$ is the constant in Theorem \ref{Th-Holder-Var}, let $u$ be the solution to \eqref{Eq-PDE} and $u_n:= u* \eta_n=n^d u*\eta(n\cdot)$.
By the definition of $\mP_x$ and It\^o's formula, we have
$$
\e^{-\lambda t} u_n (\om_t) - u_n (\om_0) = \int_0^t \e^{-\lambda s}[-\lambda u_n (\om_s) + \sL_{\kappa, b}^\a u_n(\om_s)] \d s+ \int_0^t \e^{-\lambda s} \d M^{u_n}_s,
$$
which implies
\be\label{eq-un}
u_n(x)=\mE_x \left( \int_0^\infty \e^{-\lambda t}[(\lambda u_n-\sL_{\kappa, b}^\a u_n)(\om_t)]  \d s\right)= \mE_x \left(\int_0^\infty \e^{-\lambda t} g_n(\om_t) \d t\right),
\ee
where
\be\label{eq-gn}
g_n= f*\eta_n + [(\sL_\kappa^\a u)*\eta_n-\sL_\kappa^\a (u*\eta_n)]+ [(b\cdot \nabla u)*\eta_n-b\cdot \nabla (u*\eta_n)].
\ee
Noticing that $u\in \sC^{\a+\beta}$ with $\beta>0$, we have
\begin{align*}
\begin{aligned}
&[(\sL_\kappa^\a u)*\eta_n-\sL_\kappa^\a (u*\eta_n)](x)\\
=&\int_{\R^d} \eta_n(x-y)  \d y\int_{\R^d} \delta_z^\a u(y) \frac{(\kappa(y,z)-\kappa(x,z))}{|z|^{d+\a}} \d z\\
\leq &\Lambda_3 \int_{\R^d} \eta_n(x-y) |x-y|^{\vartheta}\d y \int_{\R^d} \frac{|\delta_z^\a u(y)|}{|z|^{d+\a}} \d z\leq  C n^{-\vartheta} \|u\|_{\sC^{\a+\beta}}\to 0\ (n\to\infty).
\end{aligned}
\end{align*}
And also $[(b\cdot \nabla u)*\eta_n-b\cdot \nabla (u*\eta_n)]\to 0$ uniformly in $n$. Hence, $\{g_n\}$ is uniformly bounded and converges to $f$. Taking limit in both side of  \eqref{eq-un}, we obtain
$$
u(x)= \mE_x \left(\int_0^\infty \e^{-\lambda t} f(\om_t) \d t\right),
$$
which implies the one dimensional distribution of $\mP_x$ is unique and thus the uniqueness of $\mP_x$ follows(see \cite{Et-Ku} for details).
\epf

Next we consider the case when $\a\in (1,2)$ and $b$ is just a distribution.
\bl\label{Le-WP2}
Suppose $\a\in (1,2)$, $\kappa(x,z)$ satisfies \eqref{AS1}-\eqref{AS3} and $b\in \sC^\beta$ with $\beta\in (-(\frac{\a-1}{2}\wedge \vartheta), 0]$.  Then for each $x\in \R^d$, there is a unique
probability measure $\mP_x \in \sM^{\mu}_{\kappa, b}(x)$, for some $\mu<\beta$.
\el
\bpf
{\bf Uniqueness:} The proof is similar with the one of Lemma \ref{Le-WP1}. Suppose $-\vartheta<\mu<\beta$, $\mP_x\in \sM_{\kappa, b}^\mu(x)$, thanks to the fact $\mP_x\in\sK^{\mu}(\mD)$, we only need to show $g_n\to f$ in $\sC^\mu$, where $g_n$ is defined in \eqref{eq-gn}. Notice that $\sL_\kappa^\a u\in \sC^{\beta}$, $u*\eta_n\overset{\sC^{\a+\mu}}{\longrightarrow} u$ and by Lemma \ref{Le-Lf-Var} $\sL_\kappa^\a: \sC^{\a+\mu}\to \sC^\mu$ is bounded, we get
\begin{align*}
 &\| (\sL_\kappa^\a u)*\eta_n-\sL_\kappa^\a (u*\eta_n)\|_{\sC^\mu}\\
 \leq & \|(\sL_\kappa^\a u)*\eta_n-\sL_\kappa^\a u\|_{\sC^\mu}+\|\sL_\kappa^\a(u*\eta_n)-\sL_\kappa^\a u\|_{\sC^\mu}\to 0, \quad (n
 \to \infty).
\end{align*}
Similarly, we have
$$
\|(b\cdot \nabla u)*\eta_n-b\cdot \nabla (u*\eta_n)\|_{\sC^\mu}\to 0, \quad (n\to \infty).
$$
Thus we get $\lim_{n\to\infty}\|g_n-f\|_{\sC^\mu}= \|f*\eta_n-f\|_{\sC^\mu}=0$.

{\bf Existence:}
Let $b_n=b*\eta_n$, $\kappa_n(\cdot,z)=(\kappa(\cdot,z)*\eta_n)(\cdot)$. Let $X_t^n$ be the unique solution to the following SDE:
\begin{align*}
X^n_t=x+\int_0^t \int_{\R^d} \int_0^{\infty} z \1_{[0,\kappa_n(X_{s-}^n,z))}\widetilde N(\d r, \d z, \d s)+\int_0^t b_n(X^n_{s}) \d s,
\end{align*}
where $N$ and $\widetilde N= N^\a$ are defined at the beginning of this section. Then the probability measure $\mP_x^n=\bP\circ(X^{n}_t)^{-1}$ on $\mD$ is an element in $\sM_{\kappa_n, b_n}(x)$. For any $f\in C_b^\infty$, let $u_n^\lambda$ be the solution to
$$
\lambda u^\lambda_n-\sL^\a_{\kappa_n, b_n} u^\lambda_n = f.
$$
By It\^o's formula, for any stopping times $\tau_1\leq \tau_2$,
\begin{align*}
&u_n^\lambda(X_{\tau_2}^n)- u_n^\lambda(X_{\tau_1}^n)\\
=& \int_{\tau_1}^{\tau_2}\int_{\R^d}\int_0^{\kappa_n(X_{s-},z)}[u_n^\lambda(X_{s-}^n +z)-u_n^\lambda(X_{s-}^n )]   \widetilde N(\d r, \d z, \d s)\\
&+ \lambda \int_{\tau_1}^{\tau_2} u_n^\lambda (X_s^n ) \d s- \int_{\tau_1}^{\tau_2} f(X_{s}^n ) \d s.
\end{align*}
Hence,
\begin{align*}
\int_{\tau_1}^{\tau_2} f(X_{s}^n ) \d s =& \int_{\tau_1}^{\tau_2}\int_{\R^d}\int_0^{\kappa_n(X_{s-}^n,z)} [u_n^\lambda(X_{s-}^n +z)-u_n^\lambda(X_{s-}^n )]   \widetilde N(\d r, \d z, \d s) \\
&+ u_n^\lambda(X_{\tau_1}^n )-u_n^\lambda(X_{\tau_2}^n )+\lambda \int_{\tau_1}^{\tau_2} u_n^\lambda (X_s^n ) \d s.
\end{align*}
Denote
$$
M^n_t:=\int_{0}^{t}\int_{\R^d}\int_0^{\infty} \1_{\{r<\kappa_n(X_{s-}^n,z)\}}[u_n^\lambda(X_{s-}^n +z)-u_n^\lambda(X_{s-}^n )]   \widetilde N(\d r, \d z, \d s).
$$
By Burkholder-Davis-Gundy's inequality, we get that for any $\delta>0$, $m\in \N_+$ and bounded stopping time $\tau$,
\be\label{eq-moment-b}
\begin{aligned}
&\bE \left|\int_{\tau}^{\tau+\delta} f(X_{s}^n)  \d s\right|^{2}\\
\leq &  C_m \left\{ \bE ([M^n]_{\tau+\delta}-[M^n]_{\tau}) + \|u_n^\lambda \|_\infty^{2}+(\lambda \delta \|u^\lambda_n\|_\infty )^{2}\right\}\\
\leq & C_m \left\{ \bE ([M^{n}]_{\tau+\delta}-[M^{n}]_{\tau})  + [1+(\lambda\delta)^{2}]\|u_n^\lambda\|_\infty^{2}\right\}.
\end{aligned}
\ee
On the other hand,
\be\label{eq-Q-M}
\begin{aligned}
&[M^n]_{\tau+\delta}-[M^n]_{\tau}\\
=& \int_{\tau}^{\tau+\delta}\int_{\R^d}\int_0^\infty \1_{\{r<\kappa_n(X_{s-}^n,z)\}}[u_n^\lambda(X_{s-}^n +z)-u_n^\lambda(X_{s-}^n )]^2 N(\d r, \d z, \d s)\\
\leq & C \int_{\tau}^{\tau+\delta} \int_{\R^d}\int_0^{\Lambda_2}  (|z|^2\|\nabla u^\lambda_n\|_\infty^2\wedge \|u^\lambda_n\|_\infty^2) N(\d r, \d z, \d s) \\
=& C \int_{\tau}^{\tau+\delta} \int_{\R^d} \int_0^{\Lambda_2} g_n^\lambda(z) \widetilde N(\d r,\d z; \d s)+ C \int_{\tau}^{\tau+\delta} \int_{\R^d} \int_0^{\Lambda_2} g_n^\lambda(z) \d r\frac{\d z}{|z|^{d+\a}} \d s,
\end{aligned}
\ee
where
$$
g^\lambda_n(z):= |z|^2\|\nabla u^\lambda_n\|_\infty^2\wedge \|u^\lambda_n\|_\infty^2.
$$
By Theorem \ref{Th-Holder-Var} and interpolation, we have
\begin{align}\label{eq-Max-u}
\|u^\lambda_n\|_{\infty}\lesssim \lambda^{-\theta}\|f\|_{\sC^\mu}, \quad \|\nabla u^\lambda_n\|_{\infty}\lesssim \lambda^{\frac{1}{\a}-\theta}\|f\|_{\sC^\mu}, \quad \forall\mu\in (-(\tfrac{\a-1}{2}\wedge \vartheta), \beta],  \theta\in (0, 1+\tfrac{\mu}{\a}).
\end{align}
This yields
\be\label{eq-Max-g}
|g^\lambda_n(z)| \lesssim \|f\|_{\sC^\mu}^2 (|z|^2\lambda^{-2\theta+\frac{2}{\a}}\wedge \lambda^{-2\theta}), \quad \forall \mu\in (-(\tfrac{\a-1}{2}\wedge \vartheta), \beta],  \theta\in (0, 1+\tfrac{\mu}{\a}).
\ee
For any $\delta\leq \lambda_0^{-1}$, choosing $\lambda=\delta^{-1}$ and combining \eqref{eq-moment-b}-\eqref{eq-Max-g}, we get
\be\label{Eq-Xn-Kry}
\begin{aligned}
&\bE \left|\int^{\tau+\delta}_{\tau} f(X_{s}^n)  \d s\right|^{2}\\
\leq & C \delta \int_{\R^d} g_n^\lambda(z) \frac{\d z}{|z|^{d+\a}} +C  \|u^\lambda_n\|_\infty^2 \\
\leq& C \|f\|_{\sC^\mu}^2  \left(\delta \lambda^{-2\theta+\frac{2}{\a}}\int_{|z|<\lambda^{-1/\a}}|z|^{2-d-\a} \d z+\delta  \lambda^{-2\theta}\int_{|z|\geq \lambda^{-1/\a}} |z|^{-d-\a}\d z C+\lambda^{-2\theta}\right) \\
\leq&  C\|f\|_{\sC^\mu}^2 \delta \lambda^{1-2\theta}=C\|f\|_{\sC^\mu}^2\delta^{2\theta},
\end{aligned}
\ee
%and
%\begin{align*}
%&\bE \left|\int_{\tau}^{\tau+\delta} f(X_{s}^n  \d s\right|^{4}\\
%\leq & C\bE  \left[\int_{\tau}^{\tau+\delta} \int_{\R^d} \int_0^{\Lambda_2} g_n^\lambda(z) \widetilde N(\d r,\d z; \d s) \right]^2+C\delta^2 \left(\int_{\R^d} g_n^\lambda(z) \frac{\d z}{|z|^{d+\a}}\right)^2 +C  \|u^\lambda_n\|_\infty^4 \\
%\leq& C\delta\int_{\R^d} |g_n^\lambda(z)|^2 \frac{\d z}{|z|^{d+\a}}   +C\delta^2 \left(\int_{\R^d} g_n^\lambda(z) \frac{\d z}{|z|^{d+\a}}\right)^2 + C\|f\|_{\sC^\beta}^4 \lambda^{-4\theta} \\
%\leq & C\|f\|_{\sC^\beta}^4 \left(\delta \lambda^{-4\theta+\frac{4}{\a}}\int_{|z|<\lambda^{-1/\a}}|z|^{4-d-\a} \d z+\delta \lambda^{-4\theta} \int_{|z|\geq \lambda^{-1/\a}} |z|^{-d-\a}\d z +\delta^{4\theta} \right)\\
%\leq & C\|f\|_{\sC^\beta}^4\delta^{4\theta},
%\end{align*}
here $C$ is independent with $n$. Let $A^n_t: = \int_0^t b_n(X^n_s(x)) \d s$ and $\cT$ be the collection of all bounded stopping time. The above estimate and Burkholder-Davis-Gundy's inequality yield
\begin{align*}
&\sup_{\tau\in \cT}\bE |X_{\tau+\delta}^n -X_{\tau}^n  | \\
\leq & \sup_{\tau\in \cT} \bE \left( |A_{\tau+\delta}^n-A_{\tau}^n|  +\left| \int_\tau^{\tau+\delta} \int_{\R^d} \int_0^{\kappa_n(X_{s-}^n,z)} z \widetilde N(\d r, \d z, \d s)\right|\right)\\
\leq &\sup_{\tau\in \cT}\left(\bE \left|\int_{\tau}^{\tau+\delta} b_n (X^n_s) \d s\right|^2 \right)^{\frac{1}{2}}+ C \sup_{\tau\in \cT} \bE \left(\int_{\tau}^{\tau+\delta} \int_{\R^d} \int_0^{\Lambda_2} |z|^2 N(\d r, \d z, \d s) \right)^{\frac{1}{2}}\\
\leq& C \|b\|_{\sC^\beta} \delta^\theta+ C\sup_{\tau\in \cT} \bE \left[\int_{|z|\leq 1} |z|^2 N_\tau^\delta(\d z)\right]^{\frac{1}{2}}+ C\sup_{\tau\in \cT} \bE \left[\int_{|z|>1} |z|^2 N_\tau^\delta(\d z)\right]^{\frac{1}{2}}.
\end{align*}
where
$$
N^\delta_\tau(\d z):= \int_{\tau}^{\tau+\delta} \int_0^{\Lambda_2}N(\d r, \d z, \d s),
$$
it is not hard to see that $N^\delta_\tau$ is a Poisson random measure on $\R^d$ with intensity measure $\delta \Lambda_2 \frac{\d z}{|z|^{d+\a}}$.
Notice that for fixed $\om\in \Omega$, $N_\tau^\delta$ is a counting measure, by the elementary inequality: $(\sum_{k} |a_k|^p )^{1/p}\leq (\sum_{k} |a_k|^q )^{1/q}, \ \forall\, p\geq q>0$ and $\{a_k\}\subset \R$, we also have
$$
\left(\int_{|z|>1} |z|^2 N_\tau^\delta(\d z)\right)^{\frac{1}{2}} \leq  \int_{|z|>1} |z| N_\tau^\delta(\d z).
$$
Thus, for small $\delta\leq \lambda_0^{-1}$ we have
\begin{align*}
&\sup_{\tau\in \cT}\bE |X_{\tau+\delta}^n -X_{\tau}^n  | \\
\leq& C \|b\|_{\sC^\beta} \delta^\theta+ C\sup_{\tau\in \cT} \left[\bE \int_{|z|\leq 1} |z|^2 N_\tau^\delta(\d z)\right]^{\frac{1}{2}}+ C\sup_{\tau\in \cT} \bE \int_{|z|>1} |z| N_\tau^\delta(\d z)\\
\leq & C(\|b\|_{\sC^\beta} \delta^\theta+\delta^{\frac{1}{2}}+\delta)\lesssim \delta^{\frac{1}{2}},
\end{align*}
and consequently
\begin{align*}
\lim_{\delta\downarrow 0} \bE \sup_{\tau\in \cT} |X_{\tau+\delta}^n -X_{\tau}^n  | =0.
\end{align*}
By Aldous tightness criterion, we obtain that $\{ \mP_x^n := \bP\circ (X_t^n)^{-1}\}_{n\in \N}$ is tight. So, upon taking a subsequence, still denote by $n$, we can assume that $\mP_x^n\Rightarrow \mP_x$. By \eqref{Eq-Xn-Kry}, we also have
$$
\mE_x \left|\int^{t_1}_{t_0} f(\om_s)  \d s\right|^{2}= \lim_{n\to\infty}
\mE_x^n \left|\int^{t_1}_{t_0} f(\om_s)  \d s\right|^{2} \leq  C \|f\|_{\sC^\mu}^2 |t_1-t_0|^{2\theta},
$$
where $\mu\in (-(\tfrac{\a-1}{2}\wedge \vartheta), \beta]$ and $\theta\in (0, 1+\tfrac{\mu}{\a})$, i.e. $\mP_x\in \sK^\mu(\mD)$. Hence, by Proposition \ref{Pr-Af}, for any $f\in C_b^\infty$, we can define
$$
A^{b\cdot \nabla f}_t(\om):= \lim_{n\to\infty}\int_0^t b_n\cdot \nabla f (\om_s) \d s, \quad \mP_x-a.s..
$$
Next we verify that $\mP_x\in \sM^\mu_{\kappa, b}(x)$ with $\mu\in(-(\tfrac{\a-1}{2}\wedge \vartheta), \beta)$.
Let $\cB^0_t:=\sigma(\{\om_s: \om\in \mD, s\leq t\})$, $\cB_t=\cap_{s>t} \cB^0_s$,  $\cB=\sigma(\cup_{t\in \R_+} \cB_t)$, $\mathrm{D}_{\mP_x}:= \{t>0: \mP_x(\om_t=\om_{t-})<1\}$. For any $s, s_i, t \in \mathrm{D}_{\mP_x}$, $0\leq s_1\leq s_2\leq \cdots\leq s_k\leq s\leq t$, $f\in C_b^\infty$ and $h_1, h_2, \cdots, h_k\in C_b(\R^d)$, denote $H:=\Pi_{i=1}^k h_i(\om_{s_i})\in \cB_s$, then
\be\label{eq-mp}
\begin{aligned}
&|\mE_x [(M^f_t-M^f_s) \Pi_{i=1}^k h_i(\om_{s_i})] \\
\leq &\left| (\mE_x-\mE^n_x) \left[ f(\om_t)-f(\om_s)- \int_s^t(\sL_{\kappa_m}^\a f+b_m\cdot \nabla f) (\om_r) \d r \right] H \right|\\
&+\left| \mE_x \left[ \int_s^t (\sL_{\kappa_m}^\a -\sL_\kappa^\a) f(\om_r) \d r + \int_s^t b_m \cdot \nabla f(\om_r) \d r- (A_t^{b\cdot \nabla f}-A_s^{b\cdot \nabla f})\right] H \right| \\
&+ \left| \mE^n_x \left[ f(\om_t)-f(\om_s)- \int_s^t(\sL_{\kappa_n}^\a f+b_n\cdot \nabla f) (\om_r) \d r \right] H \right|\\
&+ \left| \mE^n_x \left[\int_s^t [(\sL_{\kappa_n}^\a-\sL^\a_{\kappa_m}) f+(b_n-b_m)\cdot \nabla f]  (\om_r) \d r \right] H \right|.
\end{aligned}
\ee
Notice that for any $m$, the first term on the right side of \eqref{eq-mp} goes to $0$ as $n$ goes to 0. Since $\mP_x\in \sK^\mu(\mD)$,  by the definition of $A^{b\cdot\nabla f}_t$, we have
\begin{align*}
&\lim_{m\to\infty}\left| \mE_x \left[ \int_s^t (\sL_{\kappa_m}^\a -\sL_\kappa^\a) f(\om_r) \d r + \int_s^t b_m \cdot \nabla f(\om_r) \d r- (A_t^{b\cdot \nabla f}-A_s^{b\cdot \nabla f})\right] H \right|\\
\leq& \Pi_{i=1}^k\|h_i\|_\infty \lim_{m\to\infty}\left( \left|\mE_x\int_s^t(\sL^\a_{\kappa_m}-\sL^\a_{\kappa})f(\om_r)\d r\right|+ \mE_x\left| \int_s^t b_m \cdot \nabla f(\om_r) \d r- (A_t^{b\cdot \nabla f}-A_s^{b\cdot \nabla f}) \right| \right)\\= &0.
\end{align*}
Similarly, the fourth term goes to $0$ uniformly in $n$ as $m$ goes to $0$. And by definition, the third term on the right side of \eqref{eq-mp} is zero. Thus, letting first $n\to\infty$ and then $m\to \infty$ on the right side of \eqref{eq-mp}, we get
$$
\mE_x [(M^f_t-M^f_s) \Pi_{i=1}^k h_i(\om_{s_i})]=0, \quad \forall s, s_i, t \in \mathrm{D}_{\mP_x}, s_i\leq s\leq t.
$$
By \cite[Lemma 7.7 of Chapter 3]{Et-Ku}, $\mathrm{D}_{\mP_x}$ is at most countable, noticing that $M^f_t$ is c\`adl\`ag under $\mP_x$, we obtain
$$
\mE_x [(M^f_t-M^f_s) \Pi_{i=1}^k h_i(\om_{s_i})]=0, \quad \forall s, s_i, t \in [0,\infty), s_i\leq s\leq t.
$$
\epf

We close this section by giving the definition of weak solution.
\bd[Weak solution]\label{ws}
Let $\beta\in\R$, $\a\in (0,2)$. We say that $(\Om, \cF, \cF_t, \bP, X, N, A)$ is a weak solution to
\be\label{Xsde}
X_t=x+\int_0^t \int_{\R^d} \int_0^{\kappa(X_{s-},z)} z N^{(\a)}(\d r, \d z, \d s)+\int_0^t b(X_{s}) \d s,
\ee
if
\begin{enumerate}
\item $(\Om, \cF, \cF_t, \bP)$ is a complete filtered probability space and $X_t,$ $ A_t$ are c\`adl\`ag processes adapted with $\cF_t$. $N$ is a Poisson random measure and for any compact set $B\subseteq \R_+\times \R^d\backslash \{0\}$,
$ N(B; t)$ is a $\cF_t$ adapted Poisson process with intensity $\int_{\R^d}\int_0^\infty \1_{B}(r,z) \d r\frac{\d z}{|z|^{d+\a}}$;
\item
$$
X_t=x+\int_0^t \int_{\R^d} \int_0^{\infty} z \1_{[0,\kappa(X_{s-},z))}N^{(\a)}(\d r, \d z, \d s)+A_t,
$$
and for any $b_n\in C_b^\infty$ and $b_n\overset{\sC^\beta}{\longrightarrow} b$, we have
$$
\int_0^t b_n(X_s) \d s{\longrightarrow} A_t
$$
in probability ${\bf P}$ uniformly over bounded time intervals;
\item there are constant $\gamma, C >0$ such that
$$
\bE \left|A_{t}-A_{s}\right|^{2} \leqslant C|t-s|^{1+\gamma}, \quad s, t \in[0, T].
$$
\end{enumerate}
\ed

Thanks to the martingale representation theorem for Poisson noise(see II.1.c on p.74 of \cite{Le-Ma}),  following  the argument in \cite[Theorem II$_{10}$]{Le-Ma} and \cite[Proposition 3.13 ]{Zh-Zh1}, we have the equivalence between martingale solution and weak solution without any surprise.

\begin{theorem}\label{Th-MP-Weak}
Let $\mathbb{P}\in\sP(\mathbb{D})$, \begin{enumerate}
\item if $\a\in (0,1]$, $b\in \sC^\beta$ with $\beta>0$, then $\mathbb{P}\in\sM_{\kappa,b}(x)$ if and only if there is a weak solution $(\Om, \cF, \cF_t, \bP, X, N, A)$ so that ${\bf P}\circ X^{-1}=\mathbb{P}$;
\item  if $\a\in (1,2)$,  $b\in \sC^\beta$ with $\beta\leq 0$, then $\mathbb{P}\in\sM_{\kappa,b}^\mu(x)$ for some $\mu<\beta$ if and only if there is a weak solution $(\Om, \cF, \cF_t, \bP, X, N, A)$  so that ${\bf P}\circ X^{-1}=\mathbb{P}\in\sK^{\mu}(\mathbb{D})$.
\end{enumerate}
\end{theorem}

\section{Regularity of density of martingale solution}
Thanks to Theorem \ref{Th-MP-Weak}, it is equivalent to consider the weak solution of \eqref{Xsde} and martingale solution associated with $\sL^\a_{\kappa,b}$. We are going to prove that the law of the weak solution  of \eqref{Xsde} has a density in some Besov space under some mild assumptions. Most results in this section are inspired by Debussche and Fournier's work \cite{De-Fo}. 

\medskip

Through out this section, we assume $\nu$ satisfies the following assumption for some $\alpha\in(0,2)$:
\ba\label{Aspt3}
\begin{itemize}
\item[(i)] $\int_{|z|\geq 1}|z|^p\nu(\d z)<\infty,\quad \forall p\in [0,\a)$,
\item[(ii)] there exists $C>0$ such that $\int_{|z|\leq a}|z|^2\nu(\d z)\leq Ca^{2-\alpha}, \quad \forall a\in(0,1]$,
\item[(iii)] there exists $c>0$ such that $\int_{|z|\leq a} |\<\xi, z\>|^2 \nu(\d z) \geq c a^{2-\a},  \  \forall \xi\in\mS^{d-1}, a\in (0,1]$. 
\end{itemize}
\ea
Define
\begin{align*}
N^{(\a)}(\d r, \d z, \d s):=\left\{
\begin{aligned}
&N(\d r, \d z, \d s) &\quad \a\in(0,1)\\
&N(\d r, \d z, \d s)-\d r \1_{B_1}(z) \nu(\d z) \d s &\quad \a=1\\
&N(\d r, \d z, \d s)-\d r\nu(\d z)\d s &\quad \a\in(1,2),
\end{aligned}
\right.
\end{align*}
where $N$ is a a Poisson random measure on $\R_+\times \R^d\times \R_+$ with intensity measure $\d r\nu(\d z)\d s$.  We also assume $Y_t$ solves the following equation:
\be\label{Eq-SDEy}
Y_t=Y_0+\int_0^t a(Y_s)\d s+\int_0^t\int_{\R^d}\int_0^\infty g(Y_{s-},z)\1_{[0,k(Y_{s-},z)]}(r) N^{(\a)}(\d r, \d z, \d s),
\ee
where $a, g, k$ are bounded measurable  functions.

\begin{lemma}\label{Le-Yst}
Assume
$$
|a(y)|\leq c_0, |g(y,z)|\leq c_2|z|, |k(y,z)|\leq \lambda_2
$$
and $Y$ solves \eqref{Eq-SDEy}. Then for all $p\in(0,\alpha)$  and $0\leq s\leq t\leq s+1$ we have
\begin{align}\label{Eq-Yst}
\bE \sup_{v\in [s,t]}|Y_{v}-Y_s|^p + \bE \sup_{v\in [s,t]}|Y_{v-}-Y_s|^p \leq C(p, c_0, c_2, \lambda_2) |t-s|^{\frac{p}{\a\vee  1}}.
\end{align}
Furthermore, if $\alpha\in(0,1)$ and $p\in[\a,1)$, then for all $0\leq s\leq t\leq s+1$ we have
\begin{align}\label{Eq-Yst11}
\bE \left(\sup_{v\in [s,t]}|Y_{v}-Y_s|^p\wedge 1\right) \leq C(p, c_0, c_2, \lambda_2) |t-s|^{p}.
\end{align}
\iffalse
 \begin{itemize}
   \item  [(1)] If $\alpha\in[1,2)$, then for all $0<p<\alpha< 2$, all $0\leq s\leq t\leq s+1$,
   \begin{align}\label{Eq-st1}
   \bE[|\tilde I_{s,t}|^{p}]&=:\bE\Big[\Big|\int_s^t\int_{|z|\leq 1}\int_0^\infty 1_{[0,\tilde \kappa(Y_{u-},z)]}(r)g(Y_{u-},z)\tilde N(\d r,\d z;\d u)\Big|^p\Big]\nonumber\\&\leq C_{p}\Lambda_2\sup_{u\in[s,t]}E\Vert g(Y_{u-},\cdot)\Vert_{\sC^1}^p(t-s)^{p/\alpha},
   \end{align}
   \begin{align}\label{Eq-Yst}
  \bE[|I_{s,t}|^{p}]&=:\bE\Big[\Big|\int_s^t\int_{|z|> 1}\int_0^\infty 1_{[0,\tilde \kappa(Y_{u-},z)]}(r)g(Y_{u-},z) N(\d r,\d z;\d u)\Big|^p\Big]\nonumber\\ &\leq C_{p}\Lambda_2\sup_{u\in[s,t]}\bE\Vert g(Y_{u-},\cdot)\Vert_{\sC^1}^p(t-s)^{p/\alpha}
   \end{align}
   and
   \begin{align}\label{xst1}
   \bE|Y_t-Y_s|^p\leq C_p \Lambda_2(t-s)^{p/\alpha}.
   \end{align}
   \item  [(2)] If $\alpha\in(0,1)$, then for all $0<p<\alpha<1$, all $0\leq s\leq t\leq s+1$,
   \begin{align}\label{st2}
   \bE[|J_{s,t}|^{p}]&=:\bE\Big[\Big|\int_s^t\int_{\mathbb{R}^d}\int_0^\infty 1_{[0,\tilde \kappa(Y_{u-},z)]}(r)g(Y_{u-},z) N(\d r,\d z;\d u)\Big|^p\Big]
   \nonumber\\& \leq C_{p}\Lambda_2\sup_{u\in[s,t]}\bE\Vert g(Y_{u-},\cdot)\Vert_{\sC^1}^p(t-s)^{p/\alpha}
   \end{align}
   and
   \begin{align}\label{xst2}
   \bE|Y_t-Y_s|^p\leq C_p\Lambda_2(t-s)^p.
   \end{align}
   \end{itemize}
 \fi

\end{lemma}

\begin{proof}
 For all $0<p<\alpha$ and $0\leq s\leq t\leq s+1$, we have: if $\alpha\in(1,2)$
 \begin{align}\label{b1}
 \bE \left|\int_s^t a(Y_u)\d u\right|^p\leq C|t-s|^p\leq C|t-s|^{\frac{p}{\alpha}},
 \end{align}
 and if $\alpha\in(0,1]$
 \begin{align}\label{b2}
 \bE\left|\int_s^t a(Y_u)\d u\right|^p\leq C|t-s|^p.
 \end{align}
Then the inequality \eqref{Eq-Yst} is a simple consequence of \eqref{b1} and  \eqref{b2} and the following inequality:
\begin{align}\label{Eq-st1}
 &\bE\left[\sup_{v\in [s,t]}\left|\int_s^{v}\int_{\R^d}\int_0^\infty \1_{[0,k(Y_{u-},z)]}(r)g(Y_{u-},z) N^{(\a)}(\d r,\d z,\d u)\right|^p\right]\nonumber\\
\leq& C(p, c_2, \lambda_2) |t-s|^{p/\alpha},
\end{align}
 for all $p\in(0,\alpha)$ and $0\leq s\leq t\leq s+1$. Actually, if $\a\in (1,2)$, write
\begin{align*}
&\int_s^{v}\int_{\R^d}\int_0^\infty \1_{[0,k(Y_{u-},z)]}(r)g(Y_{u-},z) N^{(\a)}(\d r,\d z;\d u)\\
=&I_1(v)+I_2(v):=\int_s^{v}\int_{|z|\leq |t-s|^{1/\a}}\int_0^{k(Y_{u-},z)} g(Y_{u-},z) N^{(\a)}(\d r,\d z,\d u)\\
&+ \int_s^{v}\int_{|z|>|t-s|^{1/\a}}\int_0^{k(Y_{u-},z)}g(Y_{u-},z) N^{(\a)}(\d r,\d z,\d u),
\end{align*}
For $I_1$, notice that $\frac{p}{2}<1$, by Burkholder-Davis-Gundy's inequality,
\be\label{eq-pmoment1}
\begin{aligned}
\bE\left[\sup_{v\in [s,t]}|I_1(v)|^p\right]\leq & C_p \bE\left[\left|\int_s^t\int_{|z|\leq |t-s|^{1/\alpha}}\int_0^{k(Y_{u-},z)} |g(Y_{u-},z)|^2 N(\d r,\d z;\d u)\right|^{\frac{p}{2}}\right]\\
\leq &C_p \left[ \bE \int_s^t\int_{|z|\leq |t-s|^{1/\alpha}}\int_0^{k(Y_{u-},z)} |g(Y_{u-},z)|^2 N(\d r,\d z;\d u) \right]^{\frac{p}{2}}\\
\leq & C_p\lambda_2 \left[ c_2^2 |t-s| \int_{|z|\leq |t-s|^{1/\a}}|z|^2 \nu(\d z) \right]^{\frac{p}{2}}\leq C(p, c_2, \lambda_2) |t-s|^{p/\a}.
\end{aligned}
\ee
For $I_2$, similarly, we have
\begin{align*}c_2
\bE\left[\sup_{v\in [s,t]}|I_2(v)|^p\right]\leq & C_p \bE\left[\left|\int_s^t\int_{|z|> |t-s|^{1/\alpha}}\int_0^{k(Y_{u-},z)} |g(Y_{u-},z)|^2 N(\d r,\d z;\d u)\right|^{\frac{p}{2}}\right]\\
\leq &C_pc_2^2\bE \left[ \left| \int_s^t\int_{|z|> |t-s|^{1/\alpha}} \int_0^{\lambda_2} |z|^2 N(\d r,\d z;\d u)  \right|^{\frac{p}{2}}\right].
\end{align*}
Let $N_{s,t}(\d z)= \int_s^t \int_0^{\lambda_2} N(\d r, \d z; \d u)$, then $N_{s,t}$ is a Poisson random measure with intensity $\lambda_2|t-s|\nu(\d z)$. Notice that $N_{s,t}$ is a counting measure,  by the elementary inequality: $(\sum_{k} |a_k|^p )^{1/p}\leq (\sum_{k} |a_k|^q )^{1/q}, \ \forall\, p\geq q>0,\{a_k\}\subset \R$ and  Lemma A.1 of \cite{De-Fo}, we obtain
\be\label{eq-pmoment2}
\begin{aligned}
\bE\left[\sup_{v\in [s,t]}|I_2(v)|^p\right]\leq& C_pc_2^p \bE \left[ \left| \int_{|z|> |t-s|^{1/\alpha}}|z|^2 N_{s,t}(\d z)  \right|^{\frac{p}{2}}\right] \\
\leq& C_pc_2^p\bE \int_{|z|> |t-s|^{1/\alpha}}|z|^p N_{s,t}(\d z) \\
\leq & C_pc_2^p \lambda_2 |t-s| \int_{|z|>|t-s|^{1/\a}} |z|^p \nu(\d z)\leq C(p,c_2,\lambda_2) |t-s|^{p/\a}.
\end{aligned}
\ee
Combining \eqref{eq-pmoment1} and \eqref{eq-pmoment2}, we get the desired result for $\alpha\in(1,2)$. 
By the similar argument we get that for $0<p<\a\leq1$ and $0\leq s\leq t\leq s+1$
\begin{align*}
 &\bE\left[\sup_{v\in [s,t]}\left|\int_s^{v}\int_{\R^d}\int_0^\infty \1_{[0,k(Y_{u-},z)]}(r)g(Y_{u-},z) N^{(\a)}(\d r,\d z,\d u)\right|^p\right]\nonumber\\
\leq& C(p, c_2, \lambda_2) |t-s|^{p/\alpha}.
\end{align*}
Now we only need to show that for $p\in[\alpha,1)$ and $0\leq s\leq t\leq s+1$, \eqref{Eq-Yst11} holds.
Since $\alpha p<p$, we have
\be
\label{Yp1}
\begin{aligned}
&|Y_t-Y_s|^p\wedge1\\
\leq& \left|\int_s^ta(Y_u)du\right|^p +\left|\int_s^t\int_{\R^d}\int_0^\infty \1_{[0,k(Y_{u-},z)]}(r)g(Y_{u-},z) N^{(\a)}(\d r,\d z,\d u)\right|^p\wedge1 \\
\leq&C |t-s|^p+\left|\int_s^t\int_{\R^d}\int_0^\infty \1_{[0,k(Y_{u-},z)]}(r)g(Y_{u-},z) N^{(\a)}(\d r,\d z,\d u)\right|^{\alpha p}. 
\end{aligned}
\ee
By \eqref{Eq-st1}, we get \eqref{Eq-Yst11}.
\end{proof}
\begin{lemma}\label{Le-densityY}
Suppose $\theta_i\in (0,1), i=1,2,3$ and $c_j>0, j=0,1,2,3$,
\be\label{Eq-codt-agk}
\begin{aligned}
|a(y)|\leq c_0,&\quad |a(y_1)-a(y_2)| \leq c_1 |y_1-y_2|^{\theta_1}, \\
|g(y,z)|\leq c_2|z|,&\quad |g(y_1, z)-g(y_2,z)|\leq c_3|y_1-y_2|^{\theta_2} |z|,
\end{aligned}
\ee
$k$ satisfies \eqref{AS2}, \eqref{AS3} with  $\Lambda_i$ and $\vartheta$ replaced by $\lambda_i$ and $\theta_3$, respectively.
%$Y_t$ is the solution to \eqref{Eq-SDEy}.
For any $\epsilon\in(0,t\wedge1)$, we can find a $\mathcal{F}_{t-\epsilon}$-measurable variable $V_t^\epsilon$ such that for all $p\in(0,\alpha)$
   \begin{align}\label{Eq-dif-Y}
   \bE|Y_t-Y_t^\epsilon|^p\leq C\epsilon^{\theta_0 p},
   \end{align}
   where
$$
Y^\epsilon_t= V_t^\epsilon+ \int_{t-\epsilon}^t \int_{\R^d} \int_{0}^\infty \1_{[0, k(Y_{t-\epsilon},z)]} (r) g(Y_{t-\epsilon}, z)N^{\a} (\d r, \d z, \d s),
$$
and if $\alpha\in[1,2)$, 
$$\theta_0=\frac{1}{\a}\big[(\a+\theta_1)\wedge (1+\theta_2)\wedge (1+\tfrac{\theta_3}{\a})\big],$$
 if $\alpha\in(0,1)$
$$\theta_0= \frac{1}{1-\theta_1} \wedge \frac{1}{\a}\big[(\a+\theta_1)\wedge (1+\theta_2)\wedge (1+\theta_3)\big],$$ 
   \end{lemma}

\begin{proof}
  We first prove the case when $\alpha\in(1,2)$. Take
 \begin{align*}
 V_t^\epsilon:=Y_{t-\epsilon}+\epsilon a(Y_{t-\epsilon}),
\end{align*}
then
\begin{align*}
Y_t-Y_t^\epsilon=& I_{t,\epsilon}+ J_{t,\epsilon}:=\int_{t-\epsilon}^t[a(Y_s)-a(Y_{t-\epsilon})]\d s\\
&+\int_{t-\epsilon}^t\int_{\R^d}\int_0^\infty \Big[g(Y_{s-},z)\1_{[0,k(Y_{s-},z)]}(r)-g(Y_{t-\epsilon},z)\1_{[0,k(Y_{t-\epsilon},z)]}(r)\Big]N^{\a} (\d r, \d z, \d s).
\end{align*}
For all $p\in (0,1]$,  by Jensen's inequality,
\begin{align*}
\bE [|I_{t,\epsilon} |^p] \leq \|a\|_{C^{\theta_1}}^p \bE \left( \int_{t-\epsilon}^t |Y_s-Y_{t-\epsilon}|^{\theta_1} \d s\right)^p\\
\leq \|a\|_{C^{\theta_1}}^p \left(\int_{t-\epsilon}^t  \bE   |Y_s-Y_{t-\epsilon}|^{\theta_1} \d s\right)^p\overset{\eqref{Eq-Yst}}\leq C \epsilon^{p(1+\frac{\theta_1}{\a})}.
\end{align*}
If $p\in (1, \a)$, by H\"older's inequality,
\begin{align*}
\bE [|I_{t,\epsilon} |^p] \leq \|a\|_{C^{\theta_1}}^p \bE \left[\left| \int_{t-\epsilon}^t |Y_s-Y_{t-\epsilon}|^{\theta_1} \d s\right|^p\right]\\
\leq \|a\|_{C^{\theta_1}}^p \epsilon^{p-1}\bE   \int_{t-\epsilon}^t   |Y_s-Y_{t-\epsilon}|^{p\theta_1} \d s \overset{\eqref{Eq-Yst}}\leq C \epsilon^{p(1+\frac{\theta_1}{\a})}.
\end{align*}
To sum up, for each $p\in (0,\a)$,
\be
\bE [|I_{t,\epsilon}|^p]\leq C \epsilon^{p(1+\frac{\theta_1}{\a})}.
\ee
For $J_{t,\epsilon}$,
\be \label{J}
\begin{aligned}
J_{t,\epsilon}=& \int_{t-\epsilon}^t\int_{\R^d}\int_0^{k(Y_{t-\epsilon},z)} \big[g(Y_{s-},z)-g(Y_{t-\epsilon},z) \big]N^{\a} (\d r, \d z, \d s)\\
&+ \int_{t-\epsilon}^t \int_{\R^d} \int_{k(Y_{t-\epsilon},z)}^{k(Y_{s-},z)} g(Y_{s-}, z) N^{\a} (\d r, \d z, \d s)=: J^1_{t,\epsilon}+J^2_{t,\epsilon},
\end{aligned}
\ee
where we abuse the notation $\int_u^v=-\int_v^u$ when $u>v$.  Notice that $p\in (0,\a)$, like the proof of Lemma \ref{Le-Yst}, one can see that
\begin{align*}
\bE [|J_{t,\epsilon}^1|^p]\leq &C \bE \left[\left| \int_{t-\epsilon}^t \int_{\R^d} \int_0^{\lambda_2}\big|g(Y_{s-},z)-g(Y_{t-\epsilon},z) \big|^2 N(\d r, \d z, \d s)\right|^{\frac{p}{2}}\right]\\
\leq &C\bE \left[\left| \int_{t-\epsilon}^t \int_{\R^d} \int_0^{\lambda_2}\big| Y_{s-}-Y_{t-\epsilon}\big|^{2\theta_2} |z|^2N(\d r, \d z, \d s)\right|^{\frac{p}{2}}\right]\\
\leq &C\bE  \left[ \left| \int_{t-\epsilon}^t \int_{|z|\leq \epsilon^{\frac{1}{\a}}} \int_0^{\lambda_2}\big| Y_{s-}-Y_{t-\epsilon}\big|^{2\theta_2} |z|^2N(\d r, \d z, \d s) \right|^{\frac{p}{2}}\right]\\
&+ C\bE  \left[ \left| \int_{t-\epsilon}^t \int_{|z|> \epsilon^{\frac{1}{\a}}} \int_0^{\lambda_2}\big| Y_{s-}-Y_{t-\epsilon}\big|^{2\theta_2} |z|^2N(\d r, \d z, \d s) \right|^{\frac{p}{2}}\right]\\
\leq &C \bE \left[ \sup_{s\in [t-\epsilon, t]} |Y_{s-}-Y_{t-\epsilon}|^{p\theta_2} \left(\int_{t-\epsilon}^t \int_{|z|\leq \epsilon^{\frac{1}{\a}}} \int_0^{\lambda_2}  |z|^2N(\d r, \d z, \d s) \right)^{\frac{p}{2}} \right]\\
&+ C \bE \left( \int_{t-\epsilon}^t \int_{|z|>\epsilon^{\frac{1}{\a}}} \int_0^{\lambda_2}\big| Y_{s-}-Y_{t-\epsilon}\big|^{p\theta_2} |z|^p N(\d r, \d z, \d s) \right)\\
\leq &C  \left[ \bE \sup_{s\in [t-\epsilon, t]} |Y_{s-}-Y_{t-\epsilon}|^{\a\theta_2} \right]^{\frac{p}{\a}} \left[\bE \left(\int_{t-\epsilon}^t \int_{|z|\leq \epsilon^{\frac{1}{\a}}} \int_0^{\lambda_2}  |z|^2N(\d r, \d z, \d s) \right)^{\frac{\a p}{2(\a-p)}}\right]^{1-\frac{p}{\a}} \\
&+ C \int_{t-\epsilon}^t \int_{|z|> \epsilon^{\frac{1}{\a}}} \int_0^{\lambda_2}\bE |Y_{s-}-Y_{t-\epsilon}|^{p\theta_2} |z|^p \d r\nu(\d z) \d u \overset{\eqref{Eq-Yst}}{\leq}C\epsilon^{\frac{p}{\a}(1+\theta_2)}. 
\end{align*}
Similarly, we have
\begin{align*}
\bE [|J_{t,\epsilon}^2|^p] \leq& C \bE \left [\left|\int_{t-\epsilon}^t \int_{\R^d} \int_{k(Y_{t-\epsilon},z)}^{k(Y_{s-},z)}  |z|^2 N(\d r, \d z, \d s)\right|^{\frac{p}{2}} \right]\\
\leq & C  \left| \bE \int_{t-\epsilon}^t \int_{|z|\leq \epsilon^{\frac{1}{\a}(1+\frac{\theta_3}{\a})}} \int_{k(Y_{t-\epsilon},z)}^{k(Y_{s-},z)}  |z|^2 N(\d r, \d z, \d s)\right|^{\frac{p}{2}} \\
&+ C  \bE \int_{t-\epsilon}^t \int_{|z|>\epsilon^{\frac{1}{\a}(1+\frac{\theta_3}{\a})}} \int_{k(Y_{t-\epsilon},z)}^{k(Y_{s-},z)}  |z|^p N(\d r, \d z, \d s) \\
\leq & C \left[ \int_{t-\epsilon}^t \int_{|z|\leq \epsilon^{\frac{1}{\a}(1+\frac{\theta_3}{\a})}}  \bE  |k(Y_{s-},z)-k(Y_{t-\epsilon},z)| |z|^2 \nu(\d z) \d s\right]^{\frac{p}{2}}\\
&+ C \int_{t-\epsilon}^t  \int_{|z|> \epsilon^{\frac{1}{\a}(1+\frac{\theta_3}{\a})}}  \bE  |k(Y_{s-},z)-k(Y_{t-\epsilon},z)| |z|^p \nu(\d z) \d s\\
\leq & C \left[\int_{t-\epsilon}^t\int_{|z|\leq \epsilon^{\frac{1}{\a}(1+\frac{\theta_3}{\a})}}\bE |Y_{s-}-Y_{t-\epsilon}|^{\theta_3}|z|^2 \nu(\d z) \d s\right]^{\frac{p}{2}}\\
&+C \int_{t-\epsilon}^t\int_{|z|> \epsilon^{\frac{1}{\a}(1+\frac{\theta_3}{\a})}}\bE |Y_{s-}-Y_{t-\epsilon}|^{\theta_3}|z|^p \nu(\d z) \d s
\overset{\eqref{Eq-Yst}}{\leq}C\epsilon^{\frac{p}{\a}(1+\frac{\theta_3}{\a})}.
\end{align*}
Combing the above inequalities, we get
\be\label{eq-J}
\bE [|J_{t,\epsilon}|^p]\leq C \epsilon^{\frac{p}{\a}(1+\theta_2\wedge \frac{\theta_3}{\a})}.
\ee
Thus we get \eqref{Eq-dif-Y} for $\a\in (1,2)$.
\medskip

For $\alpha\in(0,1)$, let $\delta=\epsilon^{1/(1-\theta_1)}$, $s\in [t-\epsilon, t]$,   $s_\delta=t-\epsilon+\delta\left \lfloor (s-(t-\epsilon))/\delta\right\rfloor$, here $\lfloor a \rfloor$ is the max integer less than  or equal to $a$. Consider the solution to
   \begin{align*}
   V_u^\epsilon=Y_{t-\epsilon}+\int_{t-\epsilon}^ub(V_{s_\delta}^\epsilon)\d s, \quad u\in[t-\epsilon,t].
   \end{align*}
One can see that $V_t^\epsilon$ is well defined and $\mathcal{F}_{t-\epsilon}$ measurable. Writing
  \begin{align*}
  V_u^\epsilon=Y_{t-\epsilon}+\int_{t-\epsilon}^ub(V_s^\epsilon)\d s+\int_{t-\epsilon}^u(b(V_{s_\delta}^\epsilon)-b(V_s^\epsilon))\d s.
  \end{align*}
  Then for $u\in[t-\epsilon,t]$,
  \begin{align*}
  |Y_u-V_u^\epsilon|\leq & \int_{t-\epsilon}^u|b(Y_s)-b(V_{s}^\epsilon)|\d s+\int_{t-\epsilon}^u|b(V_{s_\delta}^\epsilon)-b(V_s^\epsilon)|\d s\\
&+\int_{t-\epsilon}^t\int_{\mathbb{R}^d}\int_0^{k(Y_{s-},z)} g(Y_{s-},z) N(\d r, \d z, \d s)\\
\leq & c_1\epsilon \sup_{s\in [t-\epsilon, u]} |Y_s-V_{s}^\epsilon|^{\theta_1}+c_1 \epsilon \sup_{s\in [t-\epsilon,u]}|V_{s_\delta}^\epsilon-V_s^\epsilon|^{\theta_1} +c_2 \int_{t-\epsilon}^t\int_{\R^d} \int_0^{\lambda_2} |z| N(\d r, \d z, \d s).
  \end{align*}
  We can get
  $$
  \bE [R_{t, \epsilon}|^p] := \bE \left[ \left| \int_{t-\epsilon}^t\int_{\R^d} \int_0^{\lambda_2} |z| N(\d r, \d z, \d s)\right|^p \right] \leq C\epsilon^{p/\alpha}
  $$
  with the similar argument proving \eqref{Eq-st1}. Setting $S_{t,\epsilon}=\sup_{s\in[t-\epsilon,t]}|Y_s-V_s^\epsilon|$ and using that $b\in C^{\theta_1}(\mathbb{R}^d)$ and that $|V_s^\epsilon-V_{s_\delta}^\epsilon|\leq C\delta$, we see that
  \begin{align*}
  S_{t,\epsilon}\leq C (\epsilon S_{t,\epsilon}^{\theta_1}+\epsilon\delta^{\theta_1} +R_{t,\epsilon})=C (\epsilon S_{t,\epsilon}^{\theta_1}+\epsilon^{\frac{1}{1-\theta_1}} + R_{t,\epsilon}).
  \end{align*}
  Choosing $\epsilon$ sufficient small and using the Young inequality, we have $S_{t,\epsilon}\leq C\epsilon^{\frac{1}{1-\theta_1}}+\tfrac{\theta_1}{2}S_{t,\epsilon}+CR_{t,\epsilon}$. Thus,
  \begin{align}\label{S}
  S_{t,\epsilon}\leq CR_{t,\epsilon}+C\epsilon^{\frac{1}{1-\theta_1}}.
  \end{align}
  We finally recall that $Y_t^\epsilon=V_t^\epsilon+\int_{t-\epsilon}^t \int_{\mathbb{R}^d}\int_0^{k(Y_{t-\epsilon},z)}  g(Y_{t-\epsilon},z) N(\d r,\d z, \d s )=Y_{t-\epsilon}+\int_{t-\epsilon}^t b (V_s^\epsilon)\d s+\int_{t-\epsilon}^t (b (V_{s_\delta}^\epsilon)-b (V_s^\epsilon))\d s+\int_{t-\epsilon}^t \int_{\mathbb{R}^d}\int_0^{k(Y_{t-\epsilon},z)}  g(Y_{t-\epsilon},z) N(\d r,\d z, \d s )$
so that
\begin{align*}
|Y_t-Y_t^\epsilon|\leq &\int_{t-\epsilon}^t|b (Y_s)-b (V_{s}^\epsilon)|\d s\\&+\Big|\int_{t-\epsilon}^t\int_{\mathbb{R}^d}\int_0^\infty|\1_{[0,k(Y_{s-},z)]}(r)g(Y_{s-},z)-\1_{[0,k(Y_{t-\epsilon},z)]}(r)g(Y_{t-\epsilon},z)| N(\d r, \d z, \d s)\Big|\\&+\int_{t-\epsilon}^t |b (V_{s_\delta}^\epsilon)-b (V_s^\epsilon)|\d s=:I_{t, \epsilon}+J_{t, \epsilon}+K_{t, \epsilon}.
\end{align*}
First, by \eqref{S}
$$
I_{t, \epsilon}\leq C\int_{t-\epsilon}^t|Y_s-V_s^\epsilon|^{\theta_1}\d s\leq C(\epsilon R_{t,\epsilon}^{\theta_1}+\epsilon^{\frac{1}{1-\theta_1}}),
$$
thanks to the fact $\bE|R_{t,\epsilon}|^p\leq C\epsilon^{p/\alpha}$,
$$\bE[|I_{t, \epsilon}|^p]\leq C [\epsilon^{\frac{p}{1-\theta_1}}+\epsilon^p\bE(R_{t,\epsilon}^{p\theta_1})]\leq C [\epsilon^{\frac{p}{1-\theta_1}}+\epsilon^{p(1+\frac{\theta_1}{\alpha})}] .
$$
Next for $J_{t, \epsilon}$, by the same way of dealing with \eqref{J}, we have
\begin{align*}
\bE [|J_{t,\epsilon}|^p] 
\leq &C\bE \left[ \left|  \int_{t-\epsilon}^t \int_{|z|\leq \epsilon^{\frac{1}{\a}}} \int_0^{\lambda_2}\big| Y_{s-}-Y_{t-\epsilon}\big|^{2\theta_2} |z|^2N(\d r, \d z, \d s)\right|^{\frac{p}{2}}\right] \\
&+ C \bE \left( \int_{t-\epsilon}^t \int_{|z|>\epsilon^{\frac{1}{\a}}} \int_0^{\lambda_2}\big| Y_{s-}-Y_{t-\epsilon}\big|^{p\theta_2} |z|^p N(\d r, \d z, \d s) \right)
 \\
&+ C  \left| \bE \int_{t-\epsilon}^t \int_{|z|\leq \epsilon^{\frac{1}{\a}(1+\theta_3)}} \int_{k(Y_{t-\epsilon},z)}^{k(Y_{s-},z)}  |z|^2 N(\d r, \d z, \d s)\right|^{\frac{p}{2}} \\
&+ C  \bE \left( \int_{t-\epsilon}^t \int_{|z|>\epsilon^{\frac{1}{\a}(1+\theta_3)}} \int_{k(Y_{t-\epsilon},z)}^{k(Y_{s-},z)}  |z|^p N(\d r, \d z, \d s) \right)\\
\leq & C \bE \left[ \sup_{s\in [t-\epsilon, t]} |Y_{s-}-Y_{t-\epsilon}|^{p\theta_2} \left(\int_{t-\epsilon}^t \int_{|z|\leq \epsilon^{\frac{1}{\a}}} \int_0^{\lambda_2}  |z|^2N(\d r, \d z, \d s) \right)^{\frac{p}{2}} \right] \\&+  C \left( \int_{t-\epsilon}^t \int_{|z|>\epsilon^{\frac{1}{\a}}} \bE \big| Y_{s-}-Y_{t-\epsilon}\big|^{p\theta_2} |z|^p \nu( \d z) \d s \right)\\
&+C  \left|  \int_{t-\epsilon}^t \int_{|z|\leq\epsilon^{\frac{1+\theta_3}{\a}}} (\bE\left| Y_{s-}-Y_{t-\epsilon} \right|^{\theta_3 }\wedge 1) |z|^2 \nu( \d z) \d s\right|^{\frac{p}{2}}\\
&+ C\left( \int_{t-\epsilon}^t \int_{|z|>\epsilon^{\frac{1+\theta_3}{\a}}}  \bE(\left|Y_{s-}-Y_{t-\epsilon}\right|^{\theta_3}\wedge 1) |z|^p \nu( \d z) \d s \right)
\\
\overset{\eqref{Eq-Yst},\eqref{Eq-Yst11}}\leq&  C \epsilon^{{\frac{p}{\a}}(1+\theta_2)}+C \epsilon^{\frac{p}{\a}(1+\theta_3)}.
\end{align*}
Finally, since $b \in C^{\theta_1}(\mathbb{R}^d)$ and since $|V_s^\epsilon-V_{s_\delta}^\epsilon|\leq C\delta$, we have $K_{t,\epsilon}\leq C\epsilon\delta^{\theta_1}=C\epsilon^{\frac{1}{1-\theta_1}}$ a.s., whence $\bE[|K_{t,\epsilon}|^p]\leq C \epsilon^{\frac{p}{1-\theta_1}}$. Thus, we get \eqref{Eq-dif-Y} for $\a\in (0,1)$. The proof for $\a=1$ is similar, so we omit it here.
   \end{proof}

Now we are going to prove the regularity of the density of the process $Y_t$ defined as in \eqref{Eq-SDEy}. We first give the following lemma about the regularity of L\'evy processes.

\bl\label{Le-Density-Z}
Suppose $Z_t$ is a L\'evy process with L\'evy measure $\nu$, $\nu$ satisfies {\bf Assumption} \ref{Aspt3}. Let $p^Z_t$ denote the density of $Z_t$, then for any $s\geq 0$, $q\in [1,\infty]$ and $t\in(0,1)$,
\be\label{Eq-Density-Z}
\|p^Z_t\|_{B^s_{q,\infty}}\leq C t^{-(s+d/q')/\a},
\ee
where $C=C(s,d,\a)$, $\frac{1}{q'}=1-\frac{1}{q}$. 
\el
\bpf
Notice that 
$$
\|f\|_{B^s_{q,\infty}} = \sup_{j\geq -1} 2^{js}\|\Delta_j f\|_{q}\leq \left\|(\sum_{j\geq-1} |\Delta_j f|^2)^{1/2} \right\|_q \asymp  \|f\|_{H^s_q}, 
$$
where $H^s_q$ is the Bessel potential space. By interpolation theorem, we only need to prove
\begin{align*}
\sup_{|\a|=k}\|\p^\a p_t\|_{q}\leq C(k,d,\a)t^{-(k+d/q')/\a}, \quad k\in \N,
\end{align*}
and the above inequality is a simple consequence of \cite[Proposition 2.3]{Sc-Sz-Wa} and \cite[Lemma 1.3 and Lemma 3.3]{De-Fo}. So we complete our proof. 
\iffalse
Let $Z_t^1$(resp. $Z^2_t$) be the L\'evy process with L\'evy measure $\nu_1:=\1_{B_t^{1/\a}}\nu$(resp. $\nu_2:=\1_{B^c_{t^{1/\a}}}\nu$) and $p_t^1$(resp. $p^2_t$) be the distribution of $Z^1_t$(resp. $Z^2_t$). Choosing $l>d/2$, by Young's inequality,
\begin{align*}
\|\p^\nu p_t\|_q =& \|\p^\nu p_t^1*p_t^2\|_1\leq \|\p^\nu p^1_t\|_q \leq C \|(1+|x|^2)^{l} \p^\nu p_t^1(x) \|_{\infty}\\
\leq & C \|(1-\Delta_\xi)^l (\xi^\nu \e^{-\psi_{\nu_1}(\xi)t})\|_q.
\end{align*}
By basic calculation,
\begin{align*}
|\p_k\psi_{\nu_1} (\xi)| =& \left| \int_{|z|<{t^{1/\a}}} i z_k (\e^{i z\cdot\xi} -1) \nu(\d z)\right| \\ \leq & \left(\int_{|z|<t^{1/\a}} |z|^2 \nu(\d z) \right)^{1/2} \left( \int_{|z|<1} |z\cdot\xi|^2 \nu(\d z)\right)^{1/2}\leq C |\xi|.
\end{align*}
Similarly, for any $|\eta|>1$,
$$
|\p^\eta \psi_{\nu_1}(\xi)| \leq  \int_{|z|<1} |p_{\eta}(z) | \nu(\d z)\leq C(\eta,d),
$$
where $p_\eta$ is a polynomial. Hence,
\begin{align*}
\|\p^\nu p_1\|_q \leq & C \|(1-\Delta_\xi)^l (\xi^\nu \e^{-\psi_{\nu_1}(\xi)})\|_q \leq C \|(1+|\xi|)|\xi|^k \e^{-\psi_{\nu_1}(\xi)}\|_q\\
\leq & C \|(1+|\xi|)|\xi|^k \e^{-c|\xi|^\a}\|_q\leq C (k, d).
\end{align*}
\fi
\epf

\bl\label{Le-Density-Y}
Suppose $a, g$ satisfy \eqref{Eq-codt-agk},  $\theta_1>1-\a$ if $\a\in(0,1)$ and $|g(y,z)|\geq c_2'|z|$ for some $c_2'>0$, $k$ satisfies {\bf Assumption} \ref{Aspt1}  with  $\Lambda_i$ and $\vartheta$ replaced by $\lambda_i$ and $\theta_3$, respectively.
Then $Y_t$ has a density $p_t^Y$ and $p_t^Y  \in B^\gamma_{q,\infty}$ with $\gamma, q$ satisfying
\begin{align}\label{Eq-gamma-g}
0<\gamma < (1\wedge \a)(\a \theta_0-1), \quad 1\leq q<\frac{d}{d+\gamma-(1\wedge \a)(\a\theta_0-1)},
\end{align}
 where $\theta_0$ is the same number in Lemma \ref{Le-densityY}.
\el
\bpf
Recalling that $\cC_R=R\cdot \cC$, for $\gamma>0$ and $q\in[1,\infty]$ define
\ce
\cS^{-\gamma}_{q, j}:= \left\{\varphi\in \sS(\R^d): \hat{\varphi}\in \cC_{2^j}, \|\varphi\|_{q}\leq 2^{\gamma j} \right\}.
\de
Choose $\varphi\in \cS^{-\gamma}_{q',j}$, take the constructed process $V^\epsilon_t, Y_{t}^\epsilon$  from Lemma \ref{Le-densityY},
$$Y_{t}^\epsilon=V^\epsilon_t+\int_{t-\epsilon}^t\int_{\R^d} \int_0^\infty \1_{[0,k(Y_{t-\epsilon},z)]}(r)g(Y_{t-\epsilon},z) N^{(\a)}(\d r, \d z, \d s).$$
 By trangale inequality,
\begin{align*}
|\bE \varphi(Y_t)| \leq |\bE \varphi(Y_t^\epsilon) |+ |\bE \varphi(Y_t) -\bE \varphi(Y_t^\epsilon) | =: I^\epsilon_1(\varphi)+I^\epsilon_2(\varphi).
\end{align*}
Define
$$
Z_t^y:= \iint_{\R^d}\int_0^\infty \1_{[0,k(y,z)]}(r)g(y,z) N^{(\a)}(\d r, \d z, \d s),\quad y\in\mathbb{R}^d.
$$
Then $Z^y_t$ is a L\'evy process with L\'evy measure $\nu_y=\nu\circ [k(y,\cdot)g(y, \cdot)]^{-1}$. Under our assumptions, one can easily check that $\nu_y$ satisfies  {\bf Assumpiton} \ref{Aspt3}.   For $I_1^\epsilon(\varphi)$, recall that $V_t^\epsilon\in \sF_{t-\epsilon}$, we get
\be\label{Eq-I1}
\begin{aligned}
I_1^\epsilon(\varphi) =& |\bE \, [(\bE\  \varphi(Y^\epsilon_t)|\sF_{t-\epsilon})]|\\
=&\left|\bE \left[\bE\left( \varphi\left(V^\epsilon_t+\int_{t-\epsilon}^t\int_{\R^d} \int_0^\infty \1_{[0,k(Y_{t-\epsilon},z)]}(r)g(Y_{t-\epsilon},z) N^{(\a)}(\d r, \d z, \d s) \right)\Big|\sF_{t-\epsilon}\right)\right]\right|\nonumber\\
=&  \left|\bE \left[\bE\left( \varphi\left(u+\int_{t-\epsilon}^t\int_{\R^d} \int_0^\infty \1_{[0,k(y,z)]}(r)g(y,z) N^{(\a)}(\d r, \d z, \d s)\right)\Big|_{u=V_t^\epsilon, y=Y_{t-\epsilon}} \right)\right]\right|.
\end{aligned}
\ee
Define $\tau_u\varphi(\cdot):=\varphi(\cdot+u)$ for $u\in\mathbb{R}^d$. By Lemma \ref{Le-Density-Z} and Bernstein's inequality, for $q'=\frac{q}{q-1}$ and $s>\gamma$
\begin{align}\label{Eq-I1}
\begin{aligned}
I_1^\epsilon(\varphi)\leq &\sup_{u\in \R^d} \bE\, \varphi\left(u+\int_{t-\epsilon}^t\int_{\R^d} \int_0^\infty \1_{[0,k(y,z)]}(r)g(y,z) N^{(\a)}(\d r, \d z, \d s)\right)\\
=& \sup_{u\in \R^d} \bE \tau_u \varphi(Z^y_\epsilon) \leq C \|\varphi\|_{B^{-s}_{q',1}} \|p_\epsilon^{Z^y}\|_{B^s_{q, \infty}}\\
\leq & C 2^{(\gamma-s)j} \epsilon^{-\frac{1}{\a}(s+\frac{d}{q'})}.
\end{aligned}
\end{align}

Choose $p\in (0, 1\wedge \a)$, by Bernstein's inequality and Lemma \ref{Le-densityY},
\be\label{Eq-I2}
\begin{aligned}
I_2^\epsilon(\varphi) \leq& \|\varphi \|_{C^p} \bE |Y_t^\epsilon-Y_t|^p \leq C 2^{(p+\frac{d}{q'})j} \|\varphi\|_{q'} \ \epsilon^{\theta_0 p}\\
\leq& C 2^{(p+\gamma+\frac{d}{q'})j} \epsilon^{\theta_0 p}.
\end{aligned}
\ee
where $\theta_0$ keeps the same as in \eqref{Eq-dif-Y}. Notice that under our assumptions, $\a\theta_0>1$, for any
$$
p\in (0,1\wedge \a), \quad 0< \gamma < (\a \theta_0-1) p,
$$
we can choose $s, q, \epsilon$ such that
\be\label{Eq-s-eps}
q<\frac{d}{d+\gamma-(\a\theta_0-1)p}, \quad s=\frac{\a\theta_0p\gamma+d(p+\gamma+d/q')/q'}{\a\theta_0p-p-\gamma-d/q'},  \quad \epsilon = 2^{\frac{\a(\gamma-s)j}{s+d/q'}}.
\ee
Then combine \eqref{Eq-I1}, \eqref{Eq-I2} and \eqref{Eq-s-eps}, we get
\begin{align}\label{eq-bddE}
|\bE \varphi(Y_t)| \leq I^\epsilon _1(\varphi)+I^\epsilon _2(\varphi) \leq C,
\end{align}
where $C$ only depends on $d,\a,\theta_i, \lambda_i, c_i, \gamma, p, q$. When $\alpha\in[1,2)$, notice that $p$ can infinitely approach $1$, so we have
\begin{align*}
0<\gamma < (\a \theta_0-1), \quad 1\leq q<\frac{d}{d+1+\gamma-\a\theta_0}.
\end{align*}
When $\alpha\in(0,1)$, $p$ can infinitely approach $\a$, so
\begin{align*}
0<\gamma < \a(\a \theta_0-1), \quad 1\leq q<\frac{d}{d+\a+\gamma-\a^2\theta_0}.
\end{align*}
For any $\varphi\in B^{-\gamma}_{q', 1}$ and $j\geq -1$, define $\varphi_j= \frac{\Delta_j \varphi}{2^{\gamma j} \|\Delta_j \varphi\|_{q'}}$. Notice that $\varphi_j \in \cS^{-\gamma}_{q', j}$, by \eqref{eq-bddE}, we obtain
$$
|\bE \varphi(Y_t)| \leq \sum_{j\geq-1}|\bE \Delta_j \varphi(Y_t)| \leq \sum_{j\geq-1} |\bE \varphi_j(Y_t)| \cdot 2^{\gamma j} \|\Delta_j \varphi\|_{q'} \leq C \|\varphi\|_{B^{-\gamma}_{q',1}}.
$$
By duality, $p_t^Y\in B^{\gamma}_{q,\infty}$.
\epf
\iffalse
\begin{remark}
The above theorem is also an independent result about the existence and regularity of the density of the solution to the equation like the SDE as in \eqref{zvon} with the coefficients $\tilde b\in \sC^{\theta_1}$ for some $\theta_1\in[0,1]$, $g\in \sC^{\theta_2}(\mathbb{R}^d,  \sC^{1}(\mathbb{R}^d))$ for some $\theta_2\in(0,1]$ and $\tilde\kappa$ satisfies conditions \eqref{AS1}-\eqref{AS3}. Which is the most general case of the result in \cite{De-Fo}. Furthermore,  by the equivalence between martingale solution and weak solution,  the existence and uniqueness of solution $Y_t$ can be obtained by the existence and uniqueness of $X_t$, which we can get from the existence and uniqueness of the martingale solution proved in Lemma \ref{Le-WP1} and Lemma \ref{Le-WP2}.
\end{remark}
\fi
%Until now we get the regularity results about the density of the process $Y_t$ satisfying \eqref{Eq-SDEy} and $Y_t=\Phi(X_t)$. Which denotes Theorem \ref{Th-Density}. The following is the proof of it.

Now suppose $\kappa(x,z)$ satisfies \eqref{AS1}-\eqref{AS3} and $\max\{{0, (1-\a)}\}<\vartheta<1$, $\beta\in (1-\a, \vartheta)$ when $\alpha\in(0,1]$, $\beta\in (-(\frac{\a-1}{2}\wedge \vartheta) , 0]$ when $\alpha\in(1,2)$, and $b\in \sC^\beta$. By Theorem \ref{Th-Holder-Var}, we can fix $\lambda$ sufficient large such that $u\in\sC^{\a+\beta}$ is the unique solution   to the following resolvent equation in the distribution sense
$$
\lambda u-\sL_{\kappa,b}^\alpha u=b,
$$
and
$$
\|\nabla u\|_{L^\infty(\R^d)}\leq \frac{1}{2}.
$$
%\end{assumption}
\medskip
 Define $\Phi(x)=:u(x)+x$, then $\Phi$ is a diffeomorphism.
\bp\label{Prop-Transfer}
Under the same conditions as in Corollary \ref{Cor-MP}, the process $Y_t:=\Phi(X_t)$ satisfies the following SDE
\begin{align*}
Y_t=Y_0+\int_0^t a(Y_s)\d s&+\int_0^t\int_{\mR^d}\int_0^\infty g(Y_{s-},z)\1_{[0,k(Y_{s-},z)]}(r) N^{(\a)}(\d r, \d z, \d s),
\end{align*}
where $X_t$ is the weak solution to \eqref{Xsde}, 
\begin{align}\label{eq-a-k}
a(y):=\lambda u(\Phi^{-1}(y)), \quad k(y,z):=\kappa(\Phi^{-1}(y),z)
\end{align}
and
\begin{align}\label{eq-g}
g(y,z)=\Phi(\Phi^{-1}(y)+z)-y=u(\Phi^{-1}(y)+z)+z-u(\Phi^{-1}(y)).
\end{align}
Furthermore, we have $a\in \sC^{\a+\beta}$,
\be\label{eq-g}
c_2^{-1}|z|\leq |g(y,z)|\leq c_2|z|,\quad |g(y_1, z)-g(y_2,z)|\leq c_3|y_1-y_2|^{\a+\beta-1} |z|
\ee
and $k$ satisfies \eqref{AS1}-\eqref{AS3} with the same $\vartheta$ as $\kappa$.
\ep

\begin{proof}
With the similar argument showed in \cite[Proposition 2.7]{At-Bu-My}, applying It\^o's formula to $\Phi(x)=u(x)+x$ with respect to the process $X_t$, we get the desired conclusion.  %the equation \eqref{Eq-SDEy}  hold. From Lemma 2.4 and Lemma 2.6 in \cite{Zh-Zh1}, we know that the map $x\mapsto\Phi(x)$ is a $\mathcal{C}^1$-diffeomorphism on $\mathbb{R}^d$, which implies the left results.
\end{proof}

Now we are in the position of proving  Theorem \ref{Th-Density}.
\begin{proof}[\bf Proof of Theorem \ref{Th-Density}]
For $\a\in (0,1]$, letting $a=b$, $g(y,z)=z$, $k=\kappa$, we have $\theta_1=\beta$, $\theta_2$ can infinitely approach $1$, $\theta_3=\vartheta$. By Lemma \ref{Le-Density-Y}, we have $\a \theta_0= \a+\beta$ and 
$p^X_t \in B^{\gamma}_{q,\infty}$ with 
$$
0<\gamma < \a(\a +\beta -1), \quad 1\leq q<\frac{d}{d+\gamma-\a(\a+\beta-1)}. 
$$ 
For $\a\in (1,2)$, by  
%Lemma \ref{Le-Density-Y} and 
Proposition \ref{Prop-Transfer}, $Y_t=\Phi(X_t)$ satisfies \eqref{Eq-SDEy} and in this case the index $\theta_1$ can be taken infinitely approach $1$, $\theta_1=\alpha+\beta-1$ and $\theta_3=\vartheta$. Therefore, by Lemma \ref{Le-Density-Y},  $p_t^Y\in B^\gamma_{q,\infty}$ with
$$
0<\gamma < (\a +\beta -1)\wedge \tfrac{\vartheta}{\a}, \quad 1\leq q<\frac{d}{d+\gamma-(\a+\beta-1)\wedge \frac{\vartheta}{\a}}.
$$
This implies that there also exists a density $p_t^X$ of the distribution of $X_t$ such that $p_t^X =p_t^Y\circ \Phi \cdot\det(\nabla\Phi)$ and $p_t^X\in B^\gamma_{q,\infty}$.  Since the martingale solution $\mathbb{P}$ corresponding to SDE \eqref{Xsde} can be denoted by $\mathbb{P}=\bP\circ X$,  we get the desired result.
\end{proof}

Last we point out that Corollary \ref{Cor-FPE} is a consequence of Corollary\ref{Cor-MP}, Theorem \ref{Th-Density} and  Proposition 4.9.19 of \cite{Et-Ku}.

\section*{Acknowledgement}
We would like to thank Professor Xicheng Zhang, Professor Moritz Kassmann and Dr. Longjie Xie for useful conversations.

\end{document}